\documentclass[preprint,number]{elsarticle}
\usepackage{appendix}
\usepackage{graphicx,calc}
\usepackage{amsmath,amsthm,amssymb,amscd,stmaryrd}

\newif\ifcommentout

\def \Meas {\mathbf{M}}                 
\def \OMeas {\mathbf{N}}                 

\def \Oper {\mathbf{\Omega}}                     
\def \Dict {\mathbf{D}}         
\def \satom {\mathbf{d}}                
\def \aatom {\omega}            
\def \mdim {m}                          
\def \pdim {p}                          
\def \sdim {d}                          
\def \ddim {n}                          
\def \sparsity {k}                      
\def \cosparsity {\ell}                 
\def \x {\mathbf{x}}                    
\def \v {\mathbf{v}}                    
\def \u {\mathbf{u}}                    
\def \w {\mathbf{w}}                    
\def \y {\mathbf{y}}                    
\def \z {\mathbf{z}}                    
\def \supp {T}                          
\def \cosupp {\Lambda}          

\def \dimmax {\kappa}

\newcommand\norm[2] {\| #2 \|_{#1}}         
\newcommand\opnorm[2] {\|\mkern-2mu| #2 |\mkern-2mu\|_{#1}}         
\newcommand\innerp[2]{\langle #1, #2 \rangle}
\def \RR {\mathbb{R}}           
\def \sspace {\mathcal{V}}      
\def \aspace {\mathcal{W}}      
\def \vspan  {\operatorname{span}}
\def \vdim   {\operatorname{dim}}
\newtheorem{thm}{Theorem}
\newtheorem{remark}[thm]{Remark}

\newtheorem{lemma}[thm]{Lemma}
\newtheorem{Proposition}[thm]{Proposition}
\newtheorem{Definition}[thm]{Definition}
\newtheorem{corollary}[thm]{Corollary}
\def \Sphe {\mathbf{S}}         
\def \Id {\mathbf{Id}}          
\def \Range {\operatorname{Range}}             
\def \Null {\operatorname{Null}}
\def \sign {\operatorname{sign}}
\newcommand{\argmin}{\operatornamewithlimits{argmin}}

\def \subjectto {\quad\text{subject to}\quad}

\begin{document}
\title{The Cosparse Analysis Model and Algorithms\tnoteref{t1}}
\tnotetext[t1]{This work was supported in part by the EU through the project SMALL (Sparse Models, Algorithms and Learning for Large-Scale data), FET-Open programme, under grant number: 225913}
\author[inria]{S.~Nam}\ead{sangnam.nam@inria.fr}
\author[edinburgh]{M.~E.~Davies}\ead{Mike.Davies@ed.ac.uk}
\author[technion]{M.~Elad}\ead{elad@cs.technion.ac.il}
\author[inria]{R.~Gribonval}\ead{remi.gribonval@inria.fr}
\address[inria]{Centre de Recherche INRIA Rennes - Bretagne Atlantique,
Campus de Beaulieu, F-35042 Rennes, France}
\address[edinburgh]{School of Engineering and Electronics,
The University of Edinburgh, Edinburgh, EH9 3JL, UK}
\address[technion]{Department of Computer Science,
The Technion, Haifa 32000, Israel}

\begin{abstract}
After a decade of extensive study of the sparse representation
synthesis model, we can safely say that this is a mature and stable
field, with clear theoretical foundations, and appealing
applications. Alongside this approach, there is an {\em
analysis} counterpart model, which, despite its similarity to the
synthesis alternative, is markedly different. Surprisingly, the
analysis model did not get a similar attention, and its
understanding today is shallow and partial.

In this paper we take a closer look at the analysis approach, better
define it as a generative model for signals, and contrast it with
the synthesis one. This work proposes effective pursuit methods that
aim to solve inverse problems regularized with the analysis-model
prior, accompanied by a preliminary theoretical study of their
performance. We demonstrate the effectiveness of the analysis model
in several experiments.

\end{abstract}

\maketitle

\vspace{0.5in} \noindent {\bf Keywords:} Synthesis, Analysis, Sparse
Representations, Union of Subspaces, Pursuit Algorithms, Greedy
Algorithms, Compressed-Sensing.

\vspace{0.5in}

\section{Introduction}

Situated at the heart of signal and image processing, data models
are fundamental for stabilizing the solution of inverse problems,
and enabling various other tasks, such as compression, detection,
separation, sampling, and more. What are those models? Essentially,
a model poses a set of mathematical properties that the data is
believed to satisfy. Choosing these properties (i.e. the model)
carefully and wisely may lead to a highly effective treatment of the
signals in question and consequently to successful applications.

Throughout the years, a long series of models has been proposed and
used, exhibiting an evolution of ideas and improvements. In this
context, the past decade has been certainly the era of sparse and
redundant representations, a novel synthesis model for describing
signals \cite{DBLP:books/daglib/0025129,EladReview,Mallat:2008:WTS:1525499,DBLP:books/daglib/0025275}. Here is a
brief description of this model:

Assume that we are to model the signal $\x \in \RR^\sdim$. The
sparse and redundant synthesis model suggests that this signal could
be described as $\x = \Dict \z$, where $\Dict \in \RR^{\sdim \times
\ddim}$ is a possibly redundant dictionary ($\ddim\ge \sdim$), and
$\z \in \RR^\ddim$, the signal's representation, is assumed to be
sparse. Measuring the cardinality of non-zeros of $\z$ using the `$\ell_0$-norm',
such that $\|\z\|_0$ is the count of the non-zeros in $\z$, we
expect $\|\z\|_0$ to be much smaller than $\ddim$. Thus, the model
essentially assumes that any signal from the family of interest
could be described as a linear combination of few columns from the
dictionary $\Dict$. The name ``synthesis'' comes from the relation
$\x = \Dict \z$, with the obvious interpretation that the model
describes a way to synthesize a signal.

This model has been the focus of many papers, studying its core
theoretical properties by exploring practical numerical algorithms
for using it in practice (e.g.
\cite{ChenDonohoSaunders,MallatZhang,candes2007dantzig,DaiMilenkovicSP}),
evaluating theoretically these algorithms' performance guarantees
(e.g.
\cite{gribonval03:_spars,donohoElad,Tropp04greedis,TroppRelax,Ben-HaimEE10}),
addressing ways to obtain the dictionary from a bulk of data (e.g.
\cite{MOD,KSVD,Mairal:2010:OLM:1756006.1756008,DBLP:journals/tsp/SkrettingE10}),
and beyond all these, attacking a long
series of applications in signal and image processing with this
model, demonstrating often state-of-the-art results (e.g.
\cite{JLInpainting,KSVDDenoising,Llagostera-Casanovas:2010aa,Plumbley:2009aa}).
Today, after a decade of an extensive study along the above lines,
with nearly 4000 papers\footnote{This is a crude estimate, obtained
using ISI-Web-of-Science. By first searching Topic=(sparse and
representation and (dictionary or pursuit or sensing)), 240 papers
are obtained. Then we consider all the papers that cite the
above-found, and this results with $\approx$3900 papers.} written on
this model and related issues, we can safely say that this is a
mature and stable field, with clear theoretical foundations, and
appealing applications.

Interestingly, the synthesis model has a ``twin'' that takes an {\em
analysis} point of view.
This alternative assumes that for a signal of interest, the analyzed
vector $\Oper \x$ is expected to be sparse, where $\Oper \in
\RR^{\pdim \times \sdim }$ is a possibly redundant {\em analysis
operator} ($\pdim\ge \sdim$). Thus, we consider a signal as
belonging to the analysis model if $\|\Oper \x\|_0$ is small enough.
Common examples of analysis operators include: the shift invariant
wavelet transform $\Oper_{\mathrm{WT}}$ \cite{Mallat:2008:WTS:1525499}; the
finite difference operator $\Oper_{\mathrm{DIF}}$, which
concatenates the horizontal and vertical derivatives of an image and
is closely connected to total variation \cite{ROF92}; the curvelet
transform \cite{curvelet}, and more. Empirically, analysis models have
been successfully used for a variety of signal processing tasks such
as denoising, deblurring, and most recently compressed sensing, but this has been done with
little theoretical justification. 

It is well known by now \cite{elad07:analysisvsynthesis} that for a
square and invertible dictionary, the synthesis and the analysis
models are the same with $\Dict = \Oper^{-1}$. The models remain
similar for more general dictionaries, although then the gap between
them is unexplored. Despite the close-proximity between the two --
synthesis and analysis -- models, the first has been studied
extensively while the second has been left aside almost untouched.
In this paper we aim to bring justice to the analysis model by
addressing the following set of topics:
\begin{enumerate}
\item {\bf Cosparsity:} In Section~\ref{sec:model} we start
our discussion with a closer look at the sparse analysis model in
order to better define it as a generative model for signals. We show
that, while the synthesis model puts an emphasis on the non-zeros of
the representation vector $\z$, the analysis model draws its
strength from the zeros in the analysis vector $\Oper \x$.

\item {\bf Union of Subspaces:} Section~\ref{sec:model} is also
devoted to a comparison between the synthesis model and the analysis
one. We know that the synthesis model described above is an instance
of a wider family of models, built as a finite union of subspaces
\cite{Lu:2008ab}. By choosing all the sub-groups of columns from
$\Dict$ that could be combined linearly to generate signals, we get
an exponentially large family of low-dimensional subspaces that
cover the signals of interest. Adopting this perspective, the
analysis model can obtain a similar interpretation. How are the two
related to each other? Section~\ref{sec:model} considers this
question and proposes a few  answers.

\item {\bf Uniqueness:} We know that the {\em spark} of the
dictionary governs the uniqueness properties of sparse solutions of
the underdetermined linear system $\Dict \z = \x$ \cite{donohoElad}.
Can we derive a similar relation for the analysis case? As a
platform for studying the analysis uniqueness properties, we
consider an inverse problem of the form $ \y = \Meas \x$, where
$\Meas \in \RR^{\mdim \times \sdim}$ and $\mdim <\sdim$, and $\y \in
\RR^{\mdim}$ is a measurement vector. Put roughly (and this will be
better defined later on), assuming that $\x$ comes from the sparse
analysis model, could we claim that there is only one possible
solution $\x$ that can explain the measurement vector $\y$? 
Section~\ref{sec:uniqueness} presents this uniqueness study.

\item {\bf Pursuit Algorithms:} Armed with a deeper understanding of
the analysis model, we may ask how to efficiently find $\x$ for the
above-described linear inverse problem. As in the synthesis case, we can consider
either relaxation-based methods or greedy ones. 
In Section~\ref{sec:algorithm} we present two numerical approximation
algorithms: a greedy algorithm termed
``Greedy Analysis Pursuit'' (GAP) that resembles the Orthogonal
Matching Pursuit (OMP) \cite{MallatZhang} -- adapted to the analysis model --, and the previously
considered $\ell^1$-minimization approach
\cite{elad07:analysisvsynthesis,selesnick09:signalrestoration,Candes:2010ab}.
Section~\ref{sec:theory} accompanies 
the presentation of GAP with a theoretical study of its performance
guarantee, deriving a condition that resembles the ERC obtained for
OMP \cite{Tropp04greedis}. Similarly, we study the terms of success of
the $\ell_1$-minimization approach for the analysis model, deriving
a condition that is similar to the one obtained for the synthesis
sparse model \cite{Tropp04greedis}.

\item {\bf Tests:} In Section~\ref{sec:experiment} we demonstrate
the effectiveness of the analysis model and the pursuit algorithms
proposed in several experiments, starting from synthetic ones and
going all the way to a compressed-sensing test for an image based on
the analysis model: the Shepp Logan phantom.

\end{enumerate}

\noindent We believe that with the above set of contributions, the
cosparse
analysis model becomes a well-defined and competitive model to the
synthesis counterpart, equipped with all the necessary ingredients
for its practical use. Furthermore, this work leads to a series of
new questions that are parallel to those studied for the synthesis
model -- developing novel pursuit methods, a theoretical study of
pursuit algorithms for handling other inverse problems, training
$\Oper$ just as done for $\Dict$, and more. We discuss these and
other topics in Section~\ref{sec:final}.

\paragraph{Related Work}
Several works exist in the literature that are related to the
analysis model. The work by Elad et. al.
\cite{elad07:analysisvsynthesis} was the first to observe the
dichotomy of analysis and synthesis models for signals. Their study,
done in the context of the Maximum-A-Posteriori Probability
estimation, presented the two alternatives and explored cases of
equivalence between the two. They demonstrated a superiority of the
analysis-based approach in signal denoising. Further empirical
evidence of the effectiveness of the analysis-based approach can be
found in \cite{portilla09:analysis} and
\cite{selesnick09:signalrestoration} for signal and image
restoration. In \cite{selesnick09:signalrestoration} it was noted that
the nonzero coefficients play a different role in the analysis and
synthesis forms but the importance of
the zero coefficients for the analysis model -- which is reminiscent of signal characterizations through the zero-crossings of their undecimated wavelet transform~\cite{86995} -- was not explicitly identified.

More recently, Cand\`es et al. \cite{Candes:2010ab}
provided a theoretical study on the error when the analysis-based
$\ell_1$-minimization is used in the context of compressed sensing.
Our work is closely related to these contributions in various ways,
and we shall return to these papers when diving into the details of
our study.

\section{A Closer Look at the Cosparse Analysis Model}
\label{sec:model}

We start our discussion with the introduction of the sparse analysis model, and
the notion of cosparsity that is fundamental for its definition. We  also
describe how to interpret the analysis model as a generative one (just like the
synthesis counterpart). Finally, we consider the interpretation of the sparse
analysis and synthesis models as two manifestations of union-of-subspaces
models, and show how they are related.

\subsection{Introducing Cosparsity}

As described in the introduction, a conceptually simple model for
data would be to assume that each signal we consider can be
expressed (i.e., well-approximated) as a combination of a few
building atoms. Once we take this view, a simple synthesis model can
be thought of: First, there is a collection of the atomic signals
$\{\satom_j\}_{j=1}^n \in \RR^{\sdim}$ that we concatenate as the
columns of a dictionary, denoted by $\Dict \in
\RR^{\sdim\times\ddim}$. Here, typically $\ddim \ge \sdim$, implying
that the dictionary is redundant. Second, the signal $\x \in
\RR^{\sdim}$ can be expressed as a linear combination of some atoms
of $\Dict$, thus there exists $\z \in \RR^{\ddim}$ such that $\x =
\Dict \z$. Third and most importantly, $\x$ must lie in a low
dimensional subspace, and in order to ensure this, very few atoms
are used in the expression $\x = \Dict \z$, i.e., the number of
non-zeros $\norm{0}{\z}$ is very small. By the observation that
$\norm{0}{\z}$ is small, we say that $\x$ has a \emph{sparse
representation in $\Dict$}. The number $\sparsity = \norm{0}{\z}$ is
the \emph{sparsity} of $\x$.

Often, the validity of the above described sparse synthesis model is
demonstrated by applying a linear transform to a class of signals to
be processed and observing that most of the coefficients are close
to zero, exhibiting sparsity. In signal and image processing,
discrete transforms such as wavelet, Gabor, curvelet, contourlet,
shearlet, and others \cite{Mallat:2008:WTS:1525499,curvelet,contourlet,shearlet2005},
are of interest, and this empirical observation seems to give a good
support for the sparse synthesis model. Indeed, when aiming to claim
optimality of a given transform, this is exactly the approach taken
-- show that for a (theoretically-modeled) class of signals of
interest, the transform coefficients tend to exhibit a strong decay.
However, one cannot help but noticing that this approach of
validating the synthesis model seems to actually validate another
`similar' model; we are considering a model where the signals of
interest have {\em sparse analysis representations}. This point is
especially pronounced when the transform used is over-complete or
redundant.

Let us now look more carefully at the above mentioned model that
seems to be similar to the sparse synthesis one. First, let
$\Oper\in\RR^{\pdim\times\sdim}$ be a signal transformation or an
\emph{analysis operator}. Its rows are the row vectors 
$\{\aatom_j\}_{j=1}^p$ that will be applied to the signals. Applying
$\Oper$ to $\x$, we obtain the (analysis) representation $\Oper\x$
of $\x$. To capture various aspects of  the information in $\x$, we
typically have $\pdim \ge \sdim$. 

For simplicity, {\em unless stated
otherwise, we shall assume hereafter that all the rows of $\Oper$
are {\em in general position}, i.e., there are no non-trivial linear
dependencies among the rows}.\footnote{Put differently, we assume
that the spark of the matrix $\Oper^T$ is full, implying that every
set of $\sdim$ rows from $\Oper$ are linearly independent.}

Clearly, unless $\x = 0$, no representation $\Oper\x$ can be `very sparse',
since at least $\pdim -\sdim$ of the coefficients of $\Oper\x$ are
necessarily non-zeros. We shall put our emphasis on the number of
zeros in the representation, a quantity we will call
\emph{cosparsity}.

\begin{Definition}
The \emph{cosparsity} of a signal $\x \in \RR^{\sdim}$ with respect
to $\Oper \in \RR^{\pdim \times\sdim}$ (or simply the cosparsity of
$\x$) is defined to be:
\begin{eqnarray}
{Cosparsity:}~~~~~~\cosparsity := \pdim - \norm{0}{\Oper\x}
\end{eqnarray}
\end{Definition}

\noindent The index set of the zero entries of $\Oper\x$ is called the
\emph{cosupport} of $\x$. We say that $\x$ has \emph{cosparse representation} or
$\x$ is \emph{cosparse} when the cosparsity of $\x$ is large, where
by large we mean that $\cosparsity$ is close to $\sdim$. We will see
that, while $\cosparsity \le \sdim$ for an analysis operator in
general position, there are specific examples where $\cosparsity$
may exceed $\sdim$.

At first sight the replacement of \emph{sparsity} by
\emph{cosparsity} might appear to be mere semantics. However we will
see that this is not the case. In the synthesis model it is the columns
$\satom_j, j \in \supp$ associated with the index set $\supp$ of nonzero
coefficients that define the signal subspace. Removing columns from $\Dict$ not in $\supp$ leaves
this subspace unchanged. In contrast, it is the rows $\aatom_j$ associated
with the index set $\cosupp$ such that $\langle \aatom_j,\x \rangle = 0, j \in \cosupp$
that define the analysis subspace. In this case removing rows from
$\Oper$ for which $\langle \aatom_j,\x \rangle \neq 0$ leaves the
subspace unchanged.

From this perspective, the cosparse model is rather related to signal characterizations from the zero-crossings of their undecimated wavelet transform~\cite{86995} than to sparse wavelet expansions.


\subsection{Sparse Analysis Model as a Generative Model}

\label{sec:GenModel}

In a Bayesian context, one can think of data models as generators
for random signals from a pre-specified probability density
function. In that context, the signals that satisfy the
$\sparsity$-sparse synthesis model can be generated as follows:
First, choose $\sparsity$ columns of the dictionary $\Dict$ at
random (e.g. assuming a uniform probability). We denote the index
set chosen by $\supp$, and clearly $|\supp| = \sparsity$. Second,
form a coefficient vector $\z$ that is $\sparsity$-sparse, with
zeros outside the support $\supp$. The $\sparsity$ non-zeros in $\z$
can be chosen at random as well (e.g. Gaussian iid entries).
Finally, the signal is created by multiplying $\Dict$ to the
resulting sparse coefficient vector $\z$.

Could we adopt a similar view for the cosparse analysis model? The
answer is positive. Similar to the above, one can produce an
$\cosparsity$-cosparse signal in the following way: First, choose
$\cosparsity$ rows of the analysis operator $\Oper$ at random, and
those are denoted by an index set $\cosupp$ (thus,
$|\cosupp|=\cosparsity$). Second, form an arbitrary signal $\v$ in
$\RR^\sdim$ -- e.g., a random vector with Gaussian iid entries.
Then, project $\v$ to the orthogonal complement of the subspace
generated by the rows of $\Oper$ that are indexed by $\cosupp$, this
way getting the cosparse signal $\x$. Alternatively, one could first
find a basis for the orthogonal complement and then generate a
random coefficient vector for the basis.

This way, both models can be considered as generators of signals
that have a special structure, and clearly, the two signal
generators are different.  It is now time to ask how those two
families of signals inter-relate. In order to answer this question,
we take the union-of-subspaces point of view.


\subsection{Union-of-Subspaces Models}

\label{sec:Subspaces}

It is well known that the sparse synthesis model is a special
instance of a wider family of models called {\em
union-of-subspaces} \cite{Lu:2008ab,DBLP:journals/tit/BlumensathD09}. Given a dictionary $\Dict$, a vector $\z$ that
is exactly $\sparsity$-sparse with support $\supp$ leads to a signal
$\x = \Dict \z = \Dict_\supp \z_\supp$, a linear combination of
$\sparsity$ columns from $\Dict$. The notation $\Dict_\supp$ denotes
the sub-matrix of $\Dict$ containing only the columns indexed by
$\supp$. Denoting the subspace spanned by these columns by
$\sspace_\supp := \vspan (\satom_j, j \in \supp)$, the sparse
synthesis signals  belong to the union of all  $n \choose
\sparsity$ possible subspaces of dimension $\sparsity$,
\begin{eqnarray}\label{eq:SparseUnion}
\mbox{Sparse Synthesis Model:}~~~~~~~~\x \in
\cup_{\supp : |\supp|=\sparsity} ~~\sspace_\supp.
\end{eqnarray}

Similarly, the analysis model is associated to a union of subspaces
model as well. Given an analysis operator $\Oper$, a signal that is
exactly $\cosparsity$-cosparse with respect to the rows $\cosupp$
from $\Oper$ is simply in the orthogonal complement to these
$\cosparsity$ rows. Thus, we have\footnote{Note that the notation
$\Oper_\cosupp$ refers to restricting {\em rows} from $\Oper$
indexed by $\cosupp$, whereas in the synthesis case we have taken
the {\em columns}. We shall use this convention throughout this
paper, where from the context it should be clear whether rows or
columns are extracted.} $\Oper_\cosupp x = 0$, which implies that
$\x \in \aspace_\cosupp$, where $\aspace_\cosupp := \vspan
(\aatom_j, j \in \cosupp)^\perp = \left\{\x, \langle \aatom_j,
\x\rangle = 0, \forall j \in \cosupp \right\}$. Put differently, we may
write $\aspace_{\cosupp} = \Range(\Oper_{\cosupp}^T)^\perp =
\Null(\Oper_{\cosupp})$. Hence, cosparse analysis signals $\x$ 
belong to the union of all the $p \choose \cosparsity$ possible such
subspaces of dimension $\sdim - \cosparsity$,
\begin{eqnarray}\label{eq:CosparseUnion}
\mbox{Cosparse Analysis Model:}~~~~~~~~\x \in
\cup_{\cosupp : |\cosupp|= \cosparsity} ~~\aspace_\cosupp.
\end{eqnarray}
The following table summarizes these two unions of subspaces, where we recall that we consider $\Oper$ and $\Dict$ in general position.

\begin{center}
\begin{tabular}{|| l | l | l | l ||}
\hline \hline

Model & Subspaces & No. of Subspaces & Subspace dimension \\

\hline

Synthesis & $\sspace_\supp := \vspan (\satom_j, j \in \supp)$ & $n
\choose \sparsity$ & $\sparsity$ \\

\hline

Analysis & $\aspace_\cosupp := \vspan (\aatom_j, j \in
\cosupp)^\perp$ & $p \choose \cosparsity$ & $\sdim - \cosparsity$ \\

\hline \hline

\end{tabular}
\end{center}

What is the relation between these two union of subspaces, as
described in Equations~(\ref{eq:SparseUnion}) and
(\ref{eq:CosparseUnion})? In general, the answer is that the two are
different. An interesting way to compare between the two models is
to consider an $\cosparsity$-cosparse analysis model and a
corresponding $(\sdim-\cosparsity)$-sparse synthesis model, so that
the two have the same dimension in their subspaces.

Following this guideline, we consider first a special case where
$\cosparsity = d-1$. In such a case, the dimension of the analysis
subspaces is $\sdim - \cosparsity = 1$, and there are $\pdim \choose
\cosparsity$ of those. An equivalent synthesis union of
subspaces can be created, where $\sparsity=1$. We should construct a
dictionary $\Dict$ with $\ddim= {\pdim \choose \cosparsity}$ atoms
$\satom_j$, where each atom is the orthogonal complement to one of
the sets of $\cosparsity$ rows from $\Oper$. While the two models become
equivalent in this case, clearly $\ddim \gg \pdim$ in general, implying that
the sparse synthesis model becomes untractable since $\Dict$ becomes
too large.

By further assuming that $\pdim=\sdim$, we get that there are
exactly ${\pdim \choose \cosparsity} = {\sdim \choose \sdim -1}
=\sdim$ subspaces in the analysis union, and in this case $\ddim =
\pdim = \sdim$ as well. Furthermore, it is not hard to see that in
this case the synthesis atoms are obtained directly by a simple
inversion, $\Dict = \Oper^{-1}$.

Adopting a similar approach, considering the general case where
$\cosparsity$ is a general value (and not necessarily $\sdim -1$),
one could always construct a synthesis model that is equivalent to
the analysis one. We can compose the synthesis dictionary by simply
concatenating all the bases for the orthogonal complements to the
subspaces $\aspace_\cosupp$. The obtained dictionary will have at most
$(\sdim - \cosparsity) {\pdim \choose \cosparsity}$ atoms. However,
not all supports of size $\sparsity$ are allowed in the obtained
synthesis model, since otherwise the new sparse synthesis model will
strictly contain the cosparse analysis one. As such, the cosparse
analysis model may be viewed as a sparse synthesis model with some
structure.

Further on the comparison between the two models, it would be of
benefit to consider again the case $\sdim-\cosparsity = \sparsity$
(i.e., having the same dimensionality), assume that $\pdim = \ddim$
(i.e., having the same overcompleteness, for example with $\Oper = \Dict^{T}$), and compare the number of
subspaces amalgamated in each model. For the sake of simplicity we
consider a mild overcompleteness of $\pdim = \ddim = 2\sdim$. Denoting $H(t)
:= -t \log_2 t - (1-t) \log_2 (1-t)$, $0<t<1$, the number of
subspaces of low dimension $\sparsity \ll \sdim = \ddim/2$ in each
data model, from Stirling's approximation, roughly satisfies for large $\sdim$:
\begin{align*}
\mbox{Synthesis:}~~~~~~ \log_2 {\ddim \choose \sparsity} &\approx
\ddim \cdot H\left(\frac\sparsity\ddim\right)\quad~\approx \sparsity
\cdot \log_2 \frac{\ddim}{\sparsity}
\\
\mbox{Analysis:}~~~~~~ \log_2{ \pdim \choose \cosparsity} &\approx
\ddim \cdot H\left(\frac{\sdim-\sparsity}{\ddim} \right) \approx
\ddim \cdot H(0.5) = \ddim.
\end{align*}
More generally, unless $\sdim/\ddim \approx 1$, {\em there are much
fewer low-dimensional synthesis subspaces than the number of
analysis subspaces of the same dimension}. This is illustrated on
Figure~\ref{fig:NbOfSubspaces} when $\ddim = \pdim = 2\sdim$.
\begin{figure}[htbp]
\begin{center}
\includegraphics[width=\textwidth*1/2]{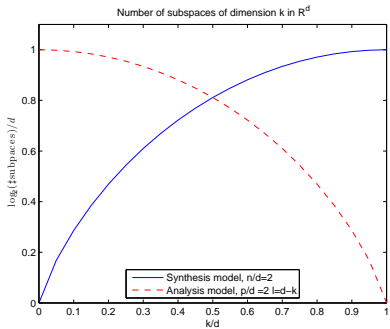}
\end{center}
\label{fig:NbOfSubspaces}
\caption{Number of subspaces of a given dimension, for $\ddim=\pdim=2\sdim$. The solid blue curve shows
the log number of subspaces for the synthesis model as the dimension of
subspaces vary, while the dashed red curve shows that for the analysis model.}
\end{figure}
This indicates a strong difference in the structure of the  two
models: The synthesis model includes very few low-dimensional
subspaces, and an increasingly large number of subspaces of higher
dimension; and the analysis model contains a combinatorial number of
low-dimensional subspaces, with fewer high dimensional subspaces.

\noindent {\bf Comment:} One must keep in mind that the huge number
of low-dimensional subspaces, though rich in terms of its
descriptive power, makes it very difficult to recover
algorithmically signals that belong to the union of those
low-dimensional subspaces or to efficiently code/sample those
signals (see the experimental results in Section~\ref{sec:performanceAnalysis}). 
This stems from the fact that in
general, it is not possible to get cosparsity $\sdim 
 \leq \cosparsity < \pdim$: any vector $\x$ that is orthogonal to
 $\sdim$ linearly independent rows of $\Oper$ must be the zero
 vector, leading to an uninformative model. One may, however, get
 cosparsities in the range $\sdim \leq \cosparsity < \pdim$ when the
 analysis operator $\Oper$ displays certain linear dependencies.
Therefore it appears to be desirable,
in the cosparse analysis model, 
to have analysis operators that exhibit highly linearly dependent
structure. We will see in Section~\ref{sec:FDOper} that a
leading example of such operators is the finite difference analysis
operator.

Another interesting point of view towards the difference between the
two models is the following: While a synthesis signal is
characterized by the support of the non-zeros in its representation
in order to define the subspace it belong to, a signal from the
analysis model is characterized by the \emph{locations of the zeros}
in its representation $\Oper \x$. The fact that this representation
may contain many non-zeroes (and especially so when $\pdim \gg
\sdim$) should be of no consequence to the efficiency of the
analysis model.


\subsection{Comparison with the Traditional Sparse Analysis model}
\label{sec:ComparisonTraditionalSparsity}

Previous work using analysis representations, both theoretical and
algorithmic, has focussed on gauging performance in terms of the more
traditional sparsity perspective.
For example, in the context of compressed sensing, recent theoretical
work \cite{Candes:2010ab} has provided performance guarantees for
minimum $\ell^1$-norm analysis representations in this light. 

The analysis operator is generally viewed as the dual
frame for a redundant synthesis dictionary so that $\Oper =
\Dict^\dagger$. This means that the analysis coefficients $\Oper \x$ provide a
consistent \emph{synthesis representation} for $\x$ in terms of the
dictionary $\Dict$, implying that the representation $\Oper \x$ is a
feasible solution to the linear system of equations $\Dict \z =
\x$.

Furthermore, if $\|\Oper \x\|_0 = \pdim-\cosparsity$, then $\Oper\x$ must be
an element of the $\sparsity$-sparse synthesis model, $\bigcup_{\supp :
  | \supp| =  \sparsity} \sspace_\supp$, with $\sparsity =
\pdim-\cosparsity$. Hence:
\begin{equation}\label{eq: model inclusion}
\{0\} \subseteq \bigcup_{\cosupp : | \cosupp | = \pdim-\sparsity} \aspace_\cosupp \subseteq
\bigcup_{\supp : | \supp| =  \sparsity} \sspace_\supp \subseteq \RR^\sdim.
\end{equation}
Of course, $\Oper \x$ is not guaranteed to be the sparsest
representation of $\x$ in terms of $\Dict$. Hence the two subspace
models are not equivalent.

Note that while in Section~\ref{sec:Subspaces} the sparsity $\sparsity$ was matched to $\sdim-\cosparsity$, here it is matched to $\pdim-\cosparsity$. The
former was used to get the same dimensions in the resulting subspaces,
while the match discussed here considers the vector $\Oper \x$ as a
candidate $\sparsity$-sparse representation.

Such a perspective treats the analysis operator as a \emph{poor man's} sparse synthesis representation. That is, for certain signals $\x$, the representation $\Oper \x$ may be reasonably sparse but is unlikely to be as sparse as, for example, the minimum $\ell^1$-norm synthesis representation\footnote{When measuring sparsity with an $\ell^{p}$ norm, $0<p\leq1$, rather than with $p=0$, it has been shown~\cite{gribonval07:_highl} that for so-called {\em localized frames} the analysis coefficients $\Oper \x$ obtained with $\Oper = \Dict^{\dagger}$ the canonical dual frame of $\Dict$ are near optimally sparse: $\|\Oper \x\|_{p} \leq C_{p} \min_{\z | \Dict\z=\x} \|\z\|_{p}$, where the constant $C_{p}$ does not depend on $\x$.}.

In the context of linear inverse problems, it is tempting to try to exploit the nesting property \eqref{eq: model inclusion} in order to derive identifiability guarantees in terms of the sparsity of the analysis coefficients $\Oper \x$. For example, in
\cite{Candes:2010ab}, the compressed sensing
recovery guarantees exploit the nesting property~\eqref{eq: model inclusion} by assuming a sufficient number of
observations to achieve a stable embedding (restricted isometry
property) for the $\sparsity$-sparse synthesis union of subspaces,
which in turn implies a stable embedding of the $(\pdim-\sparsity)$-cosparse
analysis union of subspaces.

While such an approach is of course valid, it misses a crucial
difference between the analysis and synthesis representations: they do not correspond to equivalent signal models. Treating the
two models as equivalent hides the fact that they
may be composed of subspaces with  markedly different dimensions.
The difference between these models is
highlighted in the following examples.

\subsubsection{Example: generic analysis operators, $\pdim=2\sdim$}
Assuming the rows of $\Oper$ are in general position, then when $\pdim
  \geq 2 \sdim $
  the nesting property \eqref{eq: model inclusion} is trivial but rather useless! Indeed,
  if $\sparsity < \sdim$, then the only analysis signal
  for which $\|\Oper \x\|_0 = \sparsity = \pdim - \cosparsity $ is $\x
  = 0$. Alternatively, if $\sparsity \geq \sdim$, the synthesis model
  is trivially the full space: $\bigcup_{\supp : | \supp|
    = \sparsity} \sspace_\supp = \RR^\sdim$.

\subsubsection{Example: shift invariant wavelet transform}
The shift invariant wavelet transform is a popular analysis transform in signal processing. It is particularly good for processing piecewise smooth
signals. Its inverse transform has a synthesis interpretation as the
redundant wavelet dictionary consisting of wavelet atoms with all
possible shifts.

The shift invariant wavelet transform~\cite{Mallat:2008:WTS:1525499} provides a nice example of an
analysis operator that has significant dependencies due
to the finite support of the individual wavelets. Such nontrivial
dependencies within the rows of $\Oper_{\mathrm{WT}}$ mean that the dimensions of
the (analysis or synthesis) signal subspaces are not easily characterised by either the
sparsity $\sparsity$ or the cosparsity $\cosparsity$. However the
behaviour of the model is still driven by the zero coefficients not
the nonzero ones, i.e., by the zero-crossings of the wavelet transform~\cite{86995}. By considering a particular support
set of an analysis representation $\Oper_{\mathrm{WT}} \x$ with the
shift invariant wavelet transform we can illustrate the dramatic
difference between the analysis and synthesis interpretations of the
coefficients. 

Figure~\ref{fig: cone of influence} shows the support set of the
nonzero analysis coefficients,  associated with the cone of
influence around a discontinuity in a piecewise polynomial signal of
length $128$-samples \cite{dragotti2003}, using a shift-invariant
Daubechies wavelet transform with $s=3$ vanishing moments
\cite{Mallat:2008:WTS:1525499}. For such a signal, the cone of influence at level
$J$ in a shift invariant wavelet transform contains $L_j-1$ nonzero
coefficients where $L_j$ is the length of the wavelet filter at
level $j$. Note though, the nonzero coefficients are not linearly
independent and can be elegantly described through the notion of
wavelet footprints \cite{dragotti2003}.

{\bf \em Synthesis perspective.} Interpreting the support set within
the synthesis model  implies that the signal is not particularly
sparse and needs a significant number of wavelet atoms to describe
it: in Figure~\ref{fig: cone of influence} the size of the support
set, excluding coefficients of scaling functions, is $122$. Could
the support set be significantly reduced by using a better support
selection strategy such as $\ell^1$ minimization? In practice, using
$\ell^1$ minimization, a support set of $30$ can be obtained, again
ignoring scaling coefficients.

{\bf \em Analysis perspective.} The analysis  interpretation of the
shift invariant wavelet representation relies on the examination of
the size of the analysis subspace associated with the cosupport set.
From the theory of wavelet footprints, the dimension of this
subspace is equal to the number of vanishing moments of the wavelet
filter, which in this example is only \ldots $3$, providing a much
lower dimensional signal model.

We therefore see that the analysis model has a much lower number of
degrees of freedom for this support set, leading to a significantly
more parsimonious model.

\begin{figure}[htbp]
\begin{center}
 \includegraphics[width=\textwidth*1/2]{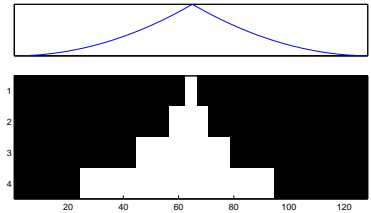}
\caption{The support set for the wavelet coefficients
  of a piecewise quadratic signal using a $J=4$ level shift invariant
  Daubechies wavelet transform with $s=3$ vanishing moments. Scaling
  coefficients are not shown. The support set contains $122$
  coefficients out of a possible $512$, yet the analysis subspace has
  a dimension of only $3$.} 
\label{fig: cone of influence}
\end{center}
\end{figure}

\subsection{Hybrid Analysis/Synthesis models?}

In this section we have demonstrated that while both the cosparse
analysis model and the sparse synthesis model can be described by a
union of subspaces these models are typically very different. We do
not argue that one is inevitably  better than the other. The value of
the model will very much depend on the problem instance. Indeed the
intrinsic difference between the models also suggests that it might be
fruitful to explore building other union of subspace models from hybrid
compositions of analysis and synthesis operators. For example, one
could imagine a signal model where $\x=\Dict \z $ through a redundant
synthesis dictionary but instead of imposing sparsity on $\z$ we
restrict $\z$ through an  additional analysis operator: $\| \Oper \z
\|_0 \leq k$. In such a case there will still be an underlying
union of subspace model but with the subspaces defined by a combination of
atoms and analysis operator constraints. A special case of this is the
split analysis model suggested in \cite{Candes:2010ab}.

\section{Uniqueness Properties}
\label{sec:uniqueness}

In the synthesis model, if a dictionary $\Dict$ is redundant, then
a given signal $\x$ can admit many synthesis
representations $\tilde\z$, i.e., $\tilde\z$ with $\Dict\tilde\z = \x$.
This makes the following type of problem interesting in the context
of the sparse signal recovery: When a signal has a sparse representation
$\z$, can there be another representation that is equally sparse or sparser?
This problem is
well-understood in terms of the so-called \emph{spark} of $\Dict$
\cite{donohoElad}, the smallest number of columns from $\Dict$ that
are linearly dependent. 

Unlike in the synthesis model, if the signal is  known, then its
analysis representation $\Oper \x$ with respect to an analysis operator $\Oper$
is completely determined. Hence, there is no inherent question of
uniqueness for the cosparse analysis model. The uniqueness question
we want to consider in this paper is in the context of the noiseless
linear inverse problem,
\begin{equation}
\label{eq:PartialMeasure}
\y = \Meas \x,
\end{equation}
where $\Meas \in \RR^{\mdim \times \sdim}$, and $\mdim < \sdim$,
implying that the measurement vector $\y \in \RR^\mdim$ is not
sufficient to fully characterize the original signal $\x\in
\RR^\sdim$. For this problem we ask: when can we assert that a
solution $\x$ with cosparsity $\cosparsity$ is the only solution
with that cosparsity or more? The problem~\eqref{eq:PartialMeasure}
(especially, with additive noise) arises ubiquitously in many
applications, and we shall focus on this problem throughout this
paper as a platform for introducing the cosparse analysis model, its
properties and behavior. Not to complicate matters unnecessarily, we
assume that all the rows of $\Meas$ are linearly independent, and we omit noise, leaving robustness analysis to further work.

For completeness of our discussion, let us return for a moment to the
synthesis model and consider the uniqueness property for the inverse
problem posed in Equation~(\ref{eq:PartialMeasure}). Assuming that
the signal's sparse representation satisfies $\x = \Dict \z$, we
have that $\y = \Meas \x = \Meas \Dict \z$. Had we known the support $\supp$
of $\z$, this linear system would have reduced to $\y =
\Meas \Dict_\supp \z_\supp$, a system of $\mdim$ equations with
$\sparsity$ unknowns. Thus, recovery of $\x$ from $\y$ is possible
only if $\sparsity \le \mdim$.

When the support of $\z$ is unknown, it is the \emph{spark} of the
compound matrix $\Meas \Dict$ that governs whether the cardinality
of $\z_\supp$ is sufficient to ensure uniqueness -- if $\sparsity=\| \z
\|_0$ is smaller than half the \emph{spark} of $\Meas \Dict$, then
necessarily $\z$ is the signal's sparsest representation. At best,
$spark(\Meas \Dict)=\mdim+1$, and then we require that the number of
measurements is at least twice the cardinality $\sparsity$. Put
formally, we require
\begin{equation}
\label{eq:Synthesis-Unique} \sparsity = \| \z \|_0 < \frac{1}{2}
spark(\Meas \Dict) \le \frac{\mdim +1}{2}.
\end{equation}
It will be interesting to contrast this requirement with the one we
will derive hereafter for the analysis model.

\subsection{Uniqueness When the Cosupport is Known}
\label{sec:uniquenessCosupportKnown}

Before we tackle the uniqueness problem for the analysis model, let
us consider an easier question: Given the observations $\y$ obtained
via a measurement matrix $\Meas$, and assuming that the cosupport
$\cosupp$ of the signal $\x$ is known, what are the sufficient
conditions for the recovery of $\x$? The answer to
this question is straightforward since $\x$ satisfies the linear
equation
\begin{equation}
\label{eq:cosparseRecoveryFormula}
\begin{bmatrix} \y \\ \mathbf{0} \\ \end{bmatrix}
= \begin{bmatrix} \Meas \\ \Oper_\cosupp \\ \end{bmatrix} \x=
\mathbf{A} \x.
\end{equation}
To be able to uniquely identify $\x$ from Equation~\eqref{eq:cosparseRecoveryFormula}, 
the matrix $\mathbf{A}$ must have a
zero null space. This is equivalent to the requirement
\begin{equation}
\label{eq:fullKnowledge}
\Null(\Oper_{\cosupp}) \cap \Null(\Meas) = 
\aspace_{\cosupp} \cap \Null(\Meas) = \{0\}.
\end{equation}

Let us now assume that $\Meas$ and $\Oper$ are mutually independent, 
in the sense that there are no nontrivial linear
dependencies among the rows of $\Meas$ and $\Oper$; 
this is a reasonable assumption because first, one should not be
measuring something that may be already available from $\Oper$,
and second, for a fixed $\Oper$, mutual independency holds true
for almost all $\Meas$ (in the Lebesgue measure).
Then, \eqref{eq:fullKnowledge} would be satisfied as soon as
$\vdim(\aspace_{\cosupp}) + \vdim(\Null(\Meas)) \le \sdim$, or
$\vdim(\aspace_{\cosupp}) \le \mdim$, since $\vdim(\Null(\Meas)) = \sdim-\mdim$.
This motivates us to define
\begin{equation}\label{dimmax definition}
\dimmax_{\Oper}(\ell) := \max_{|\cosupp| \ge \ell} ~~
\vdim(\aspace_\cosupp).
\end{equation}

The quantity $\dimmax_{\Oper}(\cosparsity)$ plays an important role in
determining the necessary and sufficient cosparsity level for the
identification of cosparse signals. 
Indeed, under the assumption of the mutual independence of $\Oper$ and $\Meas$,
a necessary and sufficient condition for the uniqueness of \emph{every}
cosparse signal given the knowledge of its cosupport $\cosupp$ of size
$\cosparsity$ is
\begin{equation}
\label{eq:KnownCosupport}
\dimmax_{\Oper}(\cosparsity) \le \mdim.
\end{equation}

\subsection{Uniqueness When the Cosupport is Unknown}
\label{sec:uniquessCosupportUnknown}

The uniqueness question that we answered above refers to the case
where the cosupport is known, but of course, in general this is not
the case. We shall assume that we may only know the cosparsity level
$\cosparsity$, which means that our uniqueness question now becomes:
what cosparsity level $\cosparsity$ guarantees that there can be
only one signal $\x$ matching a given observation $\y$?

As we have seen, the cosparse analysis model is a special case of a
general union of subspaces model. Uniqueness guarantees for missing
data problems  such as~\eqref{eq:PartialMeasure} with general union
of subspace models are covered in
\cite{Lu:2008ab,DBLP:journals/tit/BlumensathD09}. In particular
\cite{Lu:2008ab} shows that $\Meas$ is invertible on the union of
subspaces $\cup_{\gamma \in \Gamma}  S_\gamma$ \emph{if and only if}
$\Meas$ is invertible on all subspaces $S_\gamma +S_\theta$ for all
$\gamma, \theta \in \Gamma$. In the context of the analysis model
this gives the following result whose proof is a direct consequence
of the results in~\cite{Lu:2008ab}:

\begin{Proposition}[\cite{Lu:2008ab}]\label{prop: uniqueness}
Let $\cup_{\cosupp}  \aspace_\cosupp$, $|\cosupp| = \cosparsity$ be
the union of $\cosparsity$-cosparse analysis subspaces induced by
the analysis operator $\Oper$. Then the following statements are
equivalent:
\begin{enumerate}
\item If the linear system $\y = \Meas \x$ admits an $\ell$-cosparse solution, then this is the unique $\ell$-cosparse solution;
\item $\Meas$ is invertible on $\cup_{\cosupp}\aspace_\cosupp$;
\label{prop1 uniq}
\item $(\aspace_{\cosupp_1} + \aspace_{\cosupp_2}) \cap
\Null(\Meas) = 0$ for any $|\cosupp_1|,
~|\cosupp_2| \geq \cosparsity$;\label{prop1 nsp}
\end{enumerate}
\end{Proposition}

\noindent Proposition~\ref{prop: uniqueness} answers the question of
uniqueness for cosparse signals in the context of linear inverse
problems. Unfortunately, the answer we obtained still leaves us in
the dark in terms of the necessary cosparsity level or necessary
number of measurements. In order to pose a clearer condition, we use
Proposition~\ref{prop: uniqueness} from \cite{Lu:2008ab} that poses
a sharp condition on the number of measurements to guarantee
uniqueness (when $\Meas$ and $\Oper$ are mutually independent):
\begin{equation} \label{eq:nscWithF} 
\mdim \geq
\tilde{\dimmax}_{\Oper}(\cosparsity),\quad\mbox{where}\ \tilde{\dimmax}_{\Oper}(\cosparsity) := \max \left\{
\vdim(\aspace_{\cosupp_{1}}+\aspace_{\cosupp_{2}}) \ :\
|\cosupp_{i}| \geq \cosparsity, i=1,2\right\}
\end{equation}
Interestingly, a sufficient condition can also be obtained using the
quantity $\dimmax_\Oper$ defined in \eqref{dimmax definition} above,
which was observed to play 
an important role in the uniqueness result when the cosupport is
assumed to be known. Namely, we have the following result.

\begin{Proposition}\label{prop: F-L relation}
Assume that $\dimmax_{\Oper}(\cosparsity) \le \frac{\mdim}{2}$. Then for almost all $\Meas$ (wrt the Lebesgue measure), the linear inverse problem $\y =
\Meas\x$ has at most one $\cosparsity$-cosparse solution.
\end{Proposition}

\begin{proof}
Assuming the mutual independence of $\Oper$ and $\Meas$, which holds
for almost all $\Meas$, we note that  the uniqueness of $\cosparsity$ cosparse solutions holds if and only if:
\(
\dim \left(\aspace_{\cosupp_{1}} + \aspace_{\cosupp_{2}}\right) \le \mdim,
\)
whenever $|\cosupp_{i}| \geq \cosparsity$, $i=1,2$.
Assume that $\dimmax_{\Oper}(\cosparsity) \le \mdim/2$. By definition of $\dimmax_{\Oper}$, if $|\cosupp_i| \geq \cosparsity$, $i=1,2$, then
\(
\vdim(\aspace_{\cosupp_{i}}) \leq \frac{\mdim}{2},
\)
hence
\(
\vdim\left(\aspace_{\cosupp_{1}} +
\aspace_{\cosupp_{2}}\right) \leq \mdim.
\)
\end{proof}

In the synthesis model the degree to which columns are
interdependent can be partially characterized by the \emph{spark} of
$\Dict$ \cite{donohoElad} defined as the the smallest number of
columns of $\Dict$ that are linearly dependent. Here the function
$\dimmax_{\Oper}$ plays a similar role in
quantifying the interdependence between rows in the analysis model.

\begin{remark}
\label{re:Nonsharpuniqueness}
The condition
 $\dimmax_{\Oper}(\cosparsity) \le \frac{\mdim}{2}$
is in general
not necessary while condition~\eqref{eq:nscWithF}
is.
\end{remark}

There are two classes of analysis operators for which the function $\dimmax_{\Oper}$
is well-understood: analysis operators in general position and the finite difference operators.
We discuss the uniqueness results for these two classes in the following subsections.

\subsection{Analysis Operators in General Position}

It can be easily checked that
$\dimmax_{\Oper}(\cosparsity) = \max(\sdim-\cosparsity,0)$.
This enables us to quantify the exact level of cosparsity
necessary for the uniqueness guarantees:
\begin{corollary}\label{cor: gen pos uniqueness}
Let $\Oper\in\RR^{\pdim\times\sdim}$ be an analysis operator \emph{in general position}.
Then, for almost all $\mdim \times \sdim$ matrix $\Meas$, the following hold:
\begin{itemize}
\item Based on Eq.~(\ref{eq:KnownCosupport}), if $\mdim \geq \sdim-\cosparsity$,
then the equation $\y = \Meas \x$ has at most one solution with known cosupport $\cosupp$ (of cosparsity at least $\cosparsity$);
\item Based on Proposition~\ref{prop: uniqueness}, if $\mdim \geq 2(\sdim-\cosparsity)$, 
then the equation $\y = \Meas \x$ has at most one solution with cosparsity at least $\cosparsity$.
\end{itemize}
\end{corollary}

\subsection{The Finite Difference Operator}
\label{sec:FDOper}
An interesting class of analysis operators with significant linear
dependencies is the family of finite difference operators on graphs,
$\Oper_{\mathrm{DIF}}$. These are 
strongly related to TV norm minimization, popular in image processing
applications~\cite{ROF92}, and has the added benefit that we are able to quantify
the function $\dimmax_{\Oper}$ and hence the uniqueness properties of
the cosparse signal model under $\Oper_{\mathrm{DIF}}$.

We begin by considering $\Oper_{\mathrm{DIF}}$ on an arbitrary graph
before restricting our discussion to the 2D lattice associated with
image pixels. 
Consider a non-oriented graph with vertices $V$ and  edges $E
\subset V^{2}$. An edge $e$ is a pair $e = (v_{1},v_{2})$ of
connected vertices. For any vector of coefficients defined on the
vertices, $\x \in \RR^{V}$, the finite difference analysis operator
$\Oper_{\mathrm{DIF}}$ computes the collection of differences
$(x(v_{1})-x(v_{2}))$ between end-points, for all edges in the
graph. Thus, an edge $e\in E$ may be viewed as a finite difference
on $\RR^V$.

Can we estimate the function
$\dimmax_{\Oper_{\mathrm{DIF}}}(\cosparsity)$? The following shows
that it is intimately related to topological properties of the
graph. For each sub-collection  $\cosupp \subset E$ of edges, we can
define its vertex-set $V(\cosupp) \subset V$ as the collection of
vertices covered by at least one edge in $\cosupp$. The support set
$V(\cosupp)$ of $\cosupp$ can be decomposed into $J(\cosupp)$
connected components (a connected component is a set of vertices
connected to one another by a walk through vertices in $\cosupp$).
It is easy to check that a vector $\x$ belongs to the space
$\aspace_{\cosupp} = \Null(\Oper_{\cosupp})$ if and only if its
values are constant on each connected component. As a result, the
dimension of this subspace is given by
\[
\vdim (\aspace_{\cosupp})
=
|V|-|V(\cosupp)| + J(\cosupp)
\]
where the $|V|-|V(\cosupp)|$ vertices out of $V$ are associated to
arbitrary values in $\x$ that are distinct from all their neighbors,
while all entries from each of the 
$J(\cosupp)$ connected components have an arbitrary common value. It
follows that 
\begin{equation}
\label{eq:dimmaxGeneralGraph}
\dimmax_{\Oper}(\cosparsity)
= \max_{|\cosupp| \geq \cosparsity}
\Big\{|V|-|V(\cosupp)|+J(\cosupp)\Big\}
= |V|- \min_{|\cosupp| \geq \cosparsity}
\Big\{|V(\cosupp)|-J(\cosupp)\Big\}
\end{equation}
Because of the nesting of the subspaces $\aspace_{\cosupp}$, the minimum on the right hand side is achieved when $|\cosupp| = \cosparsity$.

\paragraph{Uniqueness Condition for Cosparse Images with respect to the 2D $\Oper_{\mathrm{DIF}}$}
In the abstract context of general graph the characterization~\eqref{eq:dimmaxGeneralGraph} may remain obscure, 
but can we get more concrete estimates by specializing to the 2D regular graph associated to the pixels of an $N \times N$ image?
It turns out that one can obtain relatively simple
upper and lower bounds for $\dimmax_{\Oper_{\mathrm{DIF}}}$ and hence derive an easily interpretable
uniqueness condition (see \ref{sec:proofsGraph} for a proof):
\begin{Proposition}
\label{thm:Uniqueness2DTV}
Let $\Oper_{\mathrm{DIF}}$ be the finite difference analysis operator that computes horizontal and vertical discrete derivatives of a $\sdim = N \times N$ image.  For any $\cosparsity$ we have
\begin{equation}
\label{eq:DimMaxBoundsDIF}
\sdim - \frac{\cosparsity}{2} - \sqrt{\frac{\cosparsity}{2}} -1
\le \dimmax_{\Oper_{\mathrm{DIF}}}(\cosparsity) \leq \sdim - \frac{\cosparsity}{2}.
\end{equation}
As a result, assuming that $\Meas$ is 'mutually independent' from $\Oper_{\mathrm {DIF}}$, we have:
 \begin{itemize}
\item Based on Eq.~(\ref{eq:KnownCosupport}), if $\mdim \geq \sdim-\cosparsity/2$, that is to say
\begin{equation}
\label{eq:scUniqueness2DTVKnownCosupp}
\cosparsity \geq 2\sdim-2\mdim,
\end{equation}
then the equation $\y = \Meas \x$ has at most one solution with known cosupport $\cosupp$ (of cosparsity at least $\cosparsity$);
\item Based on Proposition~\ref{prop: uniqueness}, if $\mdim \geq 2(\sdim-\cosparsity/2) = 2\sdim-\cosparsity$, that is to say
\begin{equation}
\label{eq:scUniqueness2DTV}
\cosparsity \geq 2\sdim-\mdim,
\end{equation}
then the equation $\y = \Meas \x$ has at most one solution with cosparsity at least $\cosparsity$.
\end{itemize}
\end{Proposition}

Note that as soon as the matrix $\Meas$ is associated to an
underdetermined linear system, i.e., when $\mdim < \sdim$, we need $\cosparsity \ge 2\sdim-\mdim > \sdim$ to exploit the
uniqueness guarantee~\eqref{eq:scUniqueness2DTV}.

\paragraph{The 2D $\Oper_{\mathrm{DIF}}$, Piecewise Constant
  Images, and the TV norm}

The 2D finite difference operator is closely related to the TV norm~\cite{ROF92}: the discrete TV norm of $\x$ is essentially a mixed $\ell^{2}-\ell^{1}$ norm of $\Oper_{\textrm{DIF}}\x$.
Just like its close cousin TV norm minimization, the minimization of $\|\Oper \x\|_{0}$ is particularly good at inducing piecewise constant images. We illustrate this through a worked example.

Consider the popular Shepp Logan phantom image shown in left hand
side of Figure~\ref{fig: piecewise constant images}. This particular
image has $14$ distinct connected regions of constant intensity. 
The number of non-zero coefficients in the finite difference
representation is determined by the total length (Manhattan distance)
of the boundaries between these regions. For the Shepp Logan phantom
this length is $2546$ pixel widths and thus the cosparsity is $\cosparsity = 130560-2546 =
128014$.  Furthermore, as there are no isolated pixels
with any other intensity, all pixels belong to a constant
intensity region so that
$|V(\cosupp)|=|V|$ and the cosupport has an associated subspace
dimension of:  
\begin{align*}
\vdim (\aspace_{\cosupp}) &= (|V|-|V(\cosupp)|) + J(\cosupp)\\
&= 14
\end{align*}

\begin{figure}[htbp]
\centering
\includegraphics[width=\textwidth*2/5]{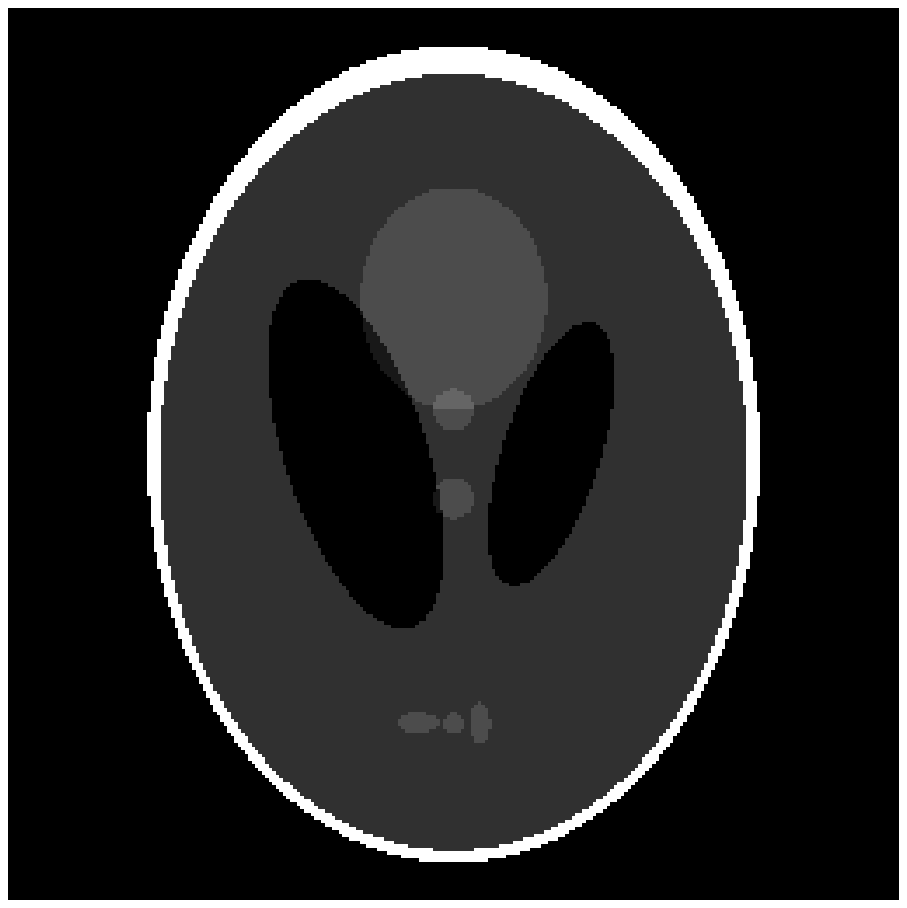}
\includegraphics[width=\textwidth*2/5]{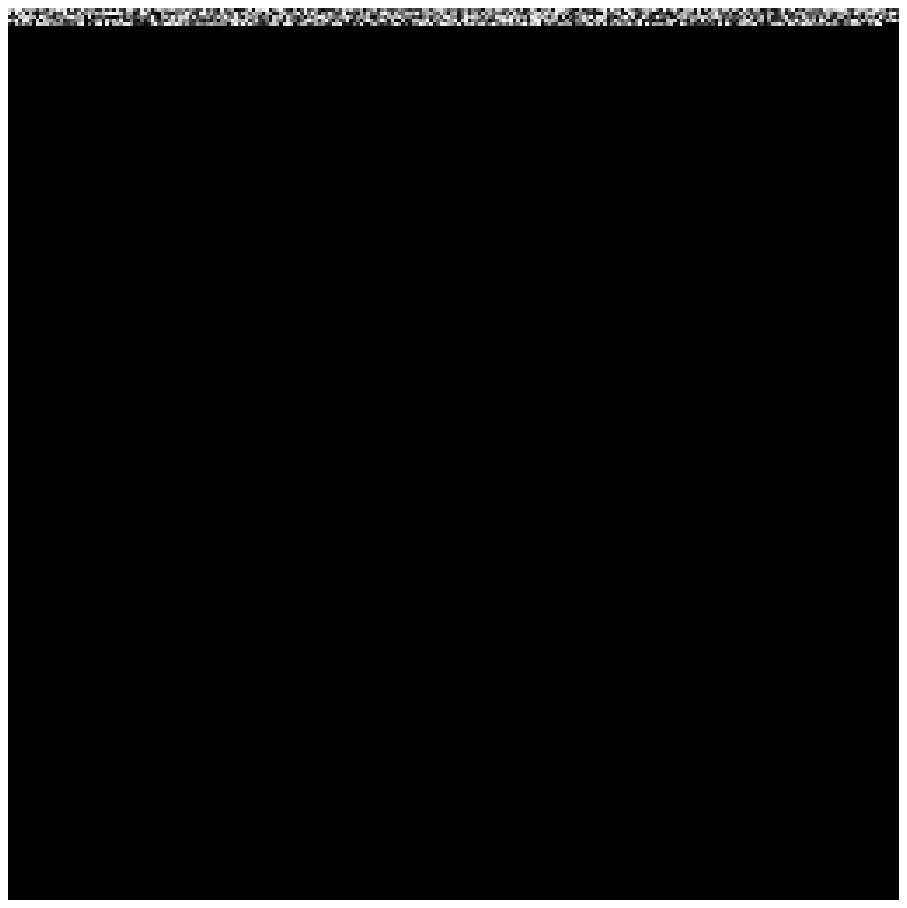}
\caption{An example of a piecewise constant image: the $256\times256$
  Shepp Logan phantom (left); and an image with the same cosparsity,
  $\cosparsity = 128014$, but whose cosupport is associated with an
  empirically maximum subspace dimension (right).}
\label{fig: piecewise constant images}
\end{figure}

In order to determine when the Shepp Logan image is the unique
solution to $\y = \Meas \x$ with maximum cosparsity it is necessary
to consider the maximum subspace dimension of all possible support sets with the
same cosparsity. This is the quantity measured by
$\dimmax_{\Oper_{\mathrm{DIF}}}(\cosparsity)$. The right hand image in
Figure~\ref{fig: piecewise 
  constant images} shows an image with equal copsparsity but 
whose support is
associated with the highest dimensional subspace we could find:
$\vdim (\aspace_{\cosupp})= 1276$. Comparing this to the bounds
given in \eqref{eq:DimMaxBoundsDIF} of Proposition~\ref{thm:Uniqueness2DTV} 
\[1270 \le \dimmax_{\Oper_{\mathrm{DIF}}}(\cosparsity) \le 1524,\] 
suggests that the lower bound is reasonably tight in this instance. 
Note, as explained in \ref{sec:proofsGraph}, this image has a
single connected subgraph, $\cosupp$, which is nearly square.
The uniqueness result from
Proposition~\ref{thm:Uniqueness2DTV} then tells us that a sufficient number
of measurements to uniquely determine the Shepp Logan 
image is given by $\mdim = 2
\dimmax_{\Oper_{\mathrm{DIF}}}(128014)$ which is somewhere
between $2552$ (if our empirical estimate is accurate) and $3048$
(worst case).

We will revisit this again in Subsection~\ref{sec:csrecovery} where we
investigate the empirical recovery performance of some practical
reconstruction algorithms.

\subsection{Overview of cosparse {\em vs} sparse models for inverse problems}
To conclude this section, Figure~\ref{fig:sketchAvsS} provides a schematic overview of analysis cosparse models {\em vs} synthesis sparse models in the context of linear inverse problems such as compressed sensing.
\begin{figure}[ht]
\begin{center}
\includegraphics[width=\textwidth*3/4]{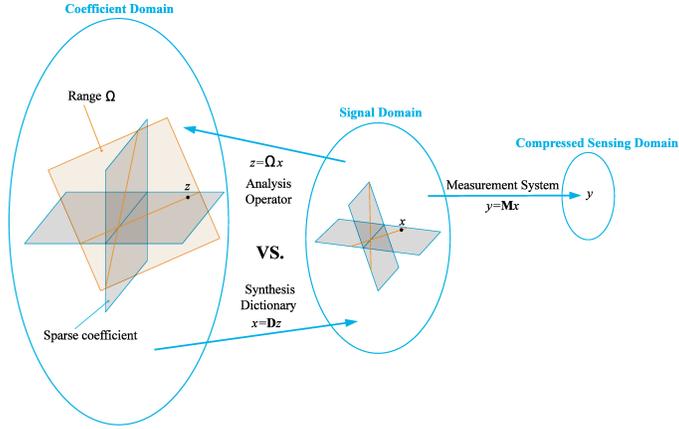}
\caption{A schematic overview of analysis cosparse {\em vs}
synthesis sparse models in relation with compressed sensing. } \label{fig:sketchAvsS}
\end{center}
\end{figure}
In the
synthesis model, the signal $\x$ is a projection (through the dictionary $\Dict$) of a high-dimensional vector $\z$
living in the union of sparse coefficient subspaces; in the analysis model, the signal lives in the pre-image by the analysis
operator $\Oper$ of the intersection between the range of $\Oper$ and this union of subspaces. 
For a given sparsity of $\z$, this is usually a set of much smaller dimensionality.

\section{Pursuit algorithms}
\label{sec:algorithm}
\newcommand{\real}{\ensuremath{\mathbb{R}}}

Having a theoretical foundation for the uniqueness of the problem
\begin{eqnarray}\label{eq:AnalInpainting1}
{\hat \x}  = \arg\min_{\x}~\|\Oper \x\|_0 ~~\mbox{subject to}~~\Meas \x = \y,
\end{eqnarray}
we turn now to the question of how to solve it: algorithms.
In this section we present two algorithms, both targeting the solution 
of problem \eqref{eq:AnalInpainting1}.
As in the uniqueness discussion, we assume that $\Meas\in \real^{\mdim \times \sdim}$, where  $\mdim < \sdim$. This implies that the equation $\Meas \x = \y$ has infinitely many possible solutions, and the term $\|\Oper \x\|_0$ introduces the analysis model to regularize the problem.

The first algorithm we present, the analysis $\ell_1$-minimization,
is well-known and widely used already in practice, see e.g.~\cite{JLInpainting,DBLP:books/daglib/0025275}.
The other algorithm we discuss is a variant of well-known greedy pursuit algorithm used for the synthesis model -- the Orthogonal Matching Pursuit (OMP) algorithm. Similar to the synthesis case, our goal is to detect the informative support of $\Oper \x$ -- 
as discussed in Section \ref{sec:uniquenessCosupportKnown},
in the analysis case, this amounts to the locations of the zeros in the vector $\Oper \x$, so as to introduce additional constraints to the underdetermined system $\Meas \x = \y$. Note that for obtaining a solution, one needs to detect at least $\sdim - \mdim $ of these zeros, and thus if $\cosparsity> \sdim - \mdim$, detection of the complete set of zeros is not mandatory. 
Of course, there can be many more possibilities to solve \eqref{eq:AnalInpainting1}
or to find approximate solutions of it. We mention a few works where some of such methods can be
found: \cite{portilla09:analysis,selesnick09:signalrestoration,DBLP:journals/mmas/CaiOS09}.

\subsection{The Analysis $\ell_1$-minimization}
Solving \eqref{eq:AnalInpainting1} can be quite difficult. In fact, the synthesis counterpart
of \eqref{eq:AnalInpainting1} is known to be NP-hard in general.
As is well-known, a very effective way to remedy this situation is to modify \eqref{eq:AnalInpainting1} and to solve:
\begin{equation}
\label{eq:AnalysisL1}
{\hat \x}  = \arg\min_{\x}~\|\Oper \x\|_1 ~~\mbox{subject to}~~\Meas \x = \y.
\end{equation}
The attractiveness of this approach comes from that \eqref{eq:AnalysisL1} is a convex problem and
hence admits computationally tractable algorithms to solve it, and that the $\ell_1$-norm
promotes high cosparsity in the solution $\hat x$.
An algorithm that targets the solution of \eqref{eq:AnalysisL1} and its convergence analysis can be found in \cite{DBLP:journals/mmas/CaiOS09}.

\subsection{The Greedy Analysis Pursuit Algorithm (GAP)}
\label{subsec:GAP}
 
The algorithm we present in this section is named Greedy Analysis Pursuit (GAP). 
As mentioned at the beginning of the section and as the name suggests,
this algorithm aims to find the cosupports of cosparse signals in a greedy
fashion.

An obvious way to find the cosupport of a cosparse signal would proceed as
follows: First, obtain a reasonable estimate of the signal from the given
information. Using the initial estimate, select a location as belonging to
the cosupport. Having this estimated part of the cosupport, we can obtain
a new estimate. One can now see that by alternating the two previous steps,
we will have identified enough locations of the cosupport to get the final
estimate.

However, the GAP works in an opposite direction and aims to detect 
the elements {\em outside} the set $\cosupp$, 
this way carving its way towards the detection of the desired cosupport. 
Therefore, the cosupport ${\hat \cosupp}$ is initialized to be the whole set $\{1,2,3,~\ldots~,p\}$, and through the iterations it is reduced towards a set of size $\cosparsity$ (or less, $\sdim -\mdim$).

Let us discuss the algorithm with some detail. 
First, the GAP uses the following initial estimate:
\begin{equation}
\label{eq:AlgoInit}
{\hat \x}_0 = \arg\min_{\x} \norm{2}{\Oper\x}^2 \subjectto \y = \Meas\x.
\end{equation} 
Not knowing the locations of the cosupport but knowing that many entries
of $\Oper\x_0$ are zero, this is a reasonable first estimate of $\x_0$.
Once we have ${\hat \x}_0$, we can view $\Oper{\hat\x}_0$ as an estimate
of $\Oper\x_0$. 
Hence, we find the location of the largest entries (in absolute value) of
$\Oper{\hat\x}_0$ and regard them as not belonging to the cosupport.
After this, we remove the corresponding rows from $\Oper$ and work with
a reduced $\Oper$.
A detailed description of the algorithm is given in Figure \ref{fig-GAP}. 

\begin{figure}[htb]
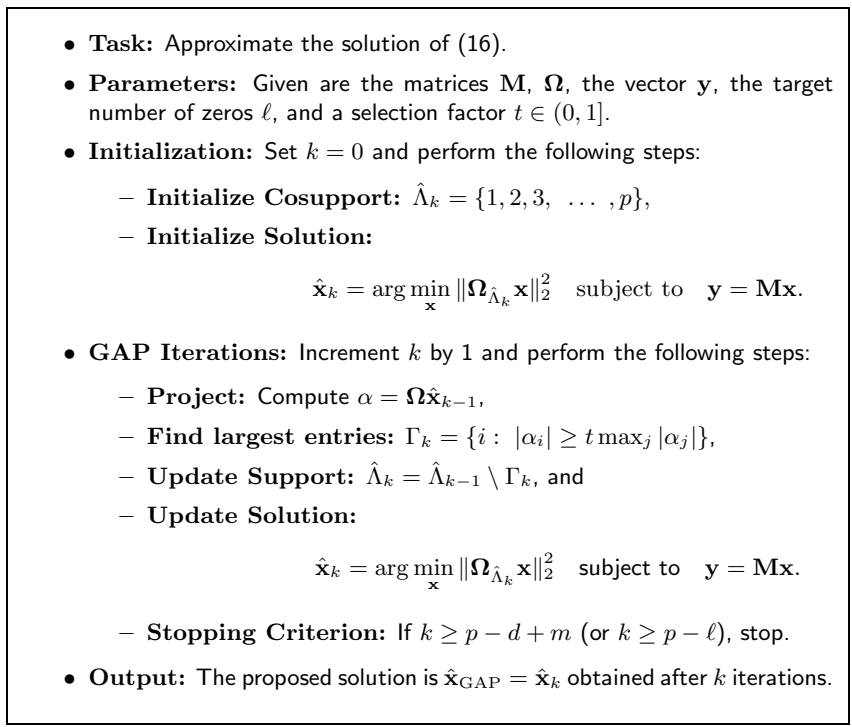

     \begin{center}
        \begin{tabular}{|c|} \hline
        \begin{minipage}[h]{\textwidth*8/9}
        \small \vspace{0.1in}
        \begin{itemize}
        \item \textsf {{\bf Task:} Approximate the solution of
            (\ref{eq:AnalInpainting1}).}
        \item \textsf {{\bf Parameters:} Given are the matrices $\Meas$, $\Oper$, the vector $\y$, the target number of zeros $\cosparsity$, and a selection factor $t \in (0,1]$.}
       \item \textsf {{\bf Initialization:} Set $k=0$ and perform the following steps:}
        \begin{itemize}
            \item{{\bf Initialize Cosupport:} ${\hat \cosupp}_k= \{1,2,3,~\ldots~,p\}$,}
            \item{{\bf Initialize Solution:} \[ \mkern100mu
            {\hat \x}_k = \arg\min_{\x} \norm{2}{\Oper_{{\hat\cosupp}_k}\x}^2 \subjectto \y = \Meas\x.\]}
        \end{itemize}
        \item\textsf {{\bf GAP Iterations:} Increment $k$ by 1 and perform the following steps:}
        \begin{itemize}
            \item \textsf {{\bf Project:} Compute $\mathbf{\alpha}=\Oper {\hat \x}_{k-1}$,}
            \item \textsf{{\bf Find largest entries:} $\Gamma_{k} = \{i:\ |\alpha_{i}| \geq t \max_{j} |\alpha_{j}|\}$,}
            \item \textsf {{\bf Update Support:} ${\hat \cosupp}_{k} = {\hat \cosupp}_{k-1} \setminus \Gamma_{k}$, and}
            \item \textsf {{\bf Update Solution:} \[ \mkern100mu
            {\hat \x}_k = \arg\min_{\x} \norm{2}{\Oper_{{\hat\cosupp}_k}\x}^2 \subjectto \y = \Meas\x.\]}
                \item \textsf {{\bf Stopping Criterion:} If $k \ge \pdim - \sdim +\mdim$ (or $k \ge \pdim - \cosparsity$), stop.}
        \end{itemize}
        \item \textsf{{\bf Output:} The proposed solution is
        $\hat\x_{\mathrm{GAP}} = \hat\x_k$ obtained after $k$ iterations.}
        \vspace{0.05in}
        \end{itemize}
        \end{minipage}
        \\\hline
        \end{tabular}
        \\ \vspace{0.1in}
        \caption{Greedy Analysis Pursuit Algorithm (GAP)}\label{fig-GAP}
     \end{center}
 \end{figure}

Some readers may notice that the GAP has similar flavors 
to the FOCUSS \cite{Gorodnitsky97sparsesignal} and the IRLS \cite{Daubechies_iterativelyreweighted}.
This is certainly true in the sense that the GAP solves constrained least squares problems
and adjusts weights as it iterates.
However, the weight adjustment in the GAP is more aggressive (removal of rows)
and binary in nature.

\paragraph{Stopping criterion / targeted sparsity}
In GAP, we debate between using the full $\cosparsity$ zeros in the product $\Oper \x$ versus a minimal and sufficient set of $\sdim -\mdim$ zeros. In between these two values, and assuming that the proper elements of $\cosupp$ have been detected, we expect the solution obtained by the algorithms to be the same, with a slightly better numerical stability for a larger number of zeros. 

Thus, an alternative stopping criterion for the GAP could be to detect whether the solution is static or the analysis coefficients of the solution are small. This way, even if the GAP made an error and removed from ${\hat \cosupp}_k$ an index that belongs to the true cosupport $\cosupp$, 
the tendency of the solution to stabilize could help in preventing the algorithm to incorporate this error into the solution. 
In fact, this criterion is used in the experiment in Section \ref{sec:experiment}. 

\paragraph{Multiple selections.} The selection factor $0<t \leq1$ allow the selection of multiple rows at once, to accelerate the algorithm by reducing the number of iterations.

\paragraph{Solving the required least squares problems}

The solution of~Eq.~\eqref{eq:AlgoInit} (and of the adjusted problems with reduced $\Oper$ at subsequent steps of the algorithm) is given analytically by
\[
\hat\x_{0} = 
\begin{bmatrix} \Meas \\ \Oper_{\hat\cosupp} \\ \end{bmatrix}^{\dagger}
\begin{bmatrix} \y \\ \mathbf{0} \\ \end{bmatrix}
= 
(\Meas^{T}\Meas + \Oper_{\hat\cosupp}^{T}\Oper_{\hat\cosupp})^{-1} \Meas^{T} \y.
\]
In practice, instead of \eqref{eq:AlgoInit}, we compute
\[
{\hat \x}_0 = \arg\min_{\x} \left\{ \norm{2}{\y - \Meas\x}^2 + \lambda \norm{2}{\Oper\x}^2 \right\}
=
\arg\min_{\x} \left\|
\begin{bmatrix} \y \\ \mathbf{0} \\ \end{bmatrix}
- \begin{bmatrix} \Meas \\ \sqrt{\lambda} \Oper \\ \end{bmatrix} \x
\right\|_{2}^{2}
\] 
for a small $\lambda > 0$, 
yielding the solution
\[
\hat\x_{0} = 
\begin{bmatrix} \Meas \\ \sqrt{\lambda} \Oper \\ \end{bmatrix}^{\dagger}
\begin{bmatrix} \y \\ \mathbf{0} \\ \end{bmatrix}
= 
(\Meas^{T}\Meas + \lambda\Oper^{T}\Oper)^{-1} \Meas^{T} \y.
\]

\section{Theoretical analysis}
\label{sec:theory}

So far, we have introduced the cosparse analysis data model, provided uniqueness results
in the context of linear inverse problems for the model, and described some algorithms
that may be used to solve such linear inverse problems to recover cosparse signals.
Before validating the algorithms and the model proposed with experimental results,
we first investigate theoretically under what conditions the proposed algorithms to solve
cosparse signal recovery~\eqref{eq:AnalInpainting1} are guaranteed to work. 
After that, we discuss the nature of the condition derived by contrasting
it to that for the synthesis model.
Further discussion including some desirable properties of $\Oper$ and $\Meas$
can be found in \ref{sec:discussionERC}.

\subsection{A Sufficient Condition for the Success of the $\ell_1$-minimization}
In the sparse synthesis framework, there is a well-known necessary and sufficient condition
called the \emph{null space property} (NSP) \cite{Donoho01uncertaintyprinciples} that guarantees the success of the synthesis
$\ell_1$-minimization
\begin{equation}
\label{eq:synthesisL1}
\hat\z_0 := \argmin_{\z} \norm{1}{\z} \subjectto \y = \mathbf{\Phi}\z
\end{equation}
to recover the sparsest solution, say $\z_0$, to $\y = \mathbf{\Phi}\z$.
To elaborate, in the case of a fixed support $\supp$, 
the $\ell_1$-minimization~\eqref{eq:synthesisL1} recovers every sparse coefficient vector
$\z_0$ supported on $\supp$ if and only if
\begin{equation}
\label{eq:synthesisNSP}
\norm{1}{\z_\supp} < \norm{1}{\z_{\supp^c}}, \quad \forall \z \in \Null(\mathbf{\Phi}),\ \z \neq 0.
\end{equation}
The NSP~\eqref{eq:synthesisNSP} cannot easily be checked but some `simpler' sufficient
conditions can be derived from it; for example, one can get a recovery condition of
\cite{Tropp04greedis} called the Exact Recovery Condition (ERC):
\begin{equation}
\label{eq:synthesisERC}
\opnorm{1\to1}{\mathbf{\Phi}^\dagger_\supp \mathbf{\Phi}_{\supp^c}} < 1,
\end{equation}
which also implies the success of greedy algorithms such as OMP \cite{Tropp04greedis}.
Note that here we used the symbol $\mathbf{\Phi}$ for an object which may be viewed
as a dictionary or a measurement matrix. Separating the data model and sampling,
we can write $\mathbf{\Phi} = \Meas\Dict$ as was done in Section~\ref{sec:uniqueness}.

One may naturally wonder: is there a condition for the cosparse analysis model that is similar to~\eqref{eq:synthesisNSP} and~\eqref{eq:synthesisERC}?
The answer to this question seems to be affirmative with some qualification
as the following two results show (the proofs are in \ref{sec:proofL1}):
\begin{thm}
\label{thm:L1NSC}
Let $\cosupp$ be a fixed cosupport. 
The analysis $\ell_1$-minimization
\begin{equation}
\label{eq:analysisL1}
\hat\x_0 := \argmin_{\x} \norm{1}{\Oper\x} \subjectto \y := \Meas\x_0 = \Meas\x
\end{equation}
recovers every $\x_0$ with cosupport $\cosupp$ as a unique minimizer if, and only if,
\begin{equation}
\label{eq:L1NSCsimple}
\sup_{\x_{\cosupp} : \Oper_{\cosupp}\x_{\cosupp}=0}\left| \innerp{\Oper_{\cosupp^c} \z}{\sign(\Oper_{\cosupp^c}\x_\cosupp)} \right|
<
\norm{1}{\Oper_{\cosupp} \z},\quad \forall \z \in \Null(\Meas),\ \z \neq 0.
\end{equation}
\end{thm}

\begin{corollary}
\label{thm:L1ERC}
Let $\OMeas^{T}$ be any $\sdim \times (\sdim-\mdim)$ basis matrix for the null space $\Null(\Meas)$, and $\cosupp$ be a fixed cosupport such that the $\cosparsity \times (\sdim-\mdim)$ matrix $\Oper_{\cosupp}\OMeas^{T}$ is of full rank $\sdim-\mdim$.  If
\begin{equation}
\label{eq:L1NSC}
\sup_{\x_\cosupp : \Oper_\cosupp\x_\cosupp = 0}
\norm{\infty}         
{
(\OMeas\Oper_\cosupp ^T)^\dagger \OMeas\Oper_{\cosupp^c}^T 
\sign(\Oper_{\cosupp^c}\x_\cosupp)}
< 1,
\end{equation}
then the analysis $\ell_1$-minimization~\eqref{eq:analysisL1}
recovers every $\x_0$ with cosupport $\cosupp$. Moreover, if
\begin{equation}
\label{eq:AL1ERC}
\opnorm{\infty\to\infty}{(\OMeas\Oper_{\cosupp}^T)^{\dagger} \OMeas \Oper_{\cosupp^c}^T }
=
\opnorm{1\to1}{\Oper_{\cosupp^c} \OMeas^{T}  (\Oper_{\cosupp} \OMeas^T)^{\dagger}}
 < 1
\end{equation}
then condition~\eqref{eq:L1NSC} holds true.
\end{corollary}
There is an apparent similarity between the analysis ERC condition~\eqref{eq:AL1ERC} above and its standard synthesis counterpart~\eqref{eq:synthesisERC}, yet there are some subtle differences between the two that will be highlighted in Section~\ref{sec:comparisonERC}.

\subsection{A Sufficient Condition for the Success of the GAP}
There is an interesting parallel between the synthesis 
ERC~\eqref{eq:synthesisERC} and its analysis version in 
Corollary~\ref{thm:L1ERC}; namely, the analysis ERC condition~\eqref{eq:AL1ERC} 
also implies the success of the GAP algorithm, as we will now show. 

From the way GAP algorithm works, we can guarantee that it will perform a
correct elimination at the first step if 
the largest analysis coefficients of $\Oper_{\cosupp^c} \hat \x_0$ of the first estimate $\hat \x_0$
are larger than the largest of $\Oper_{\cosupp} \hat \x_0$
where $\cosupp$ denotes the true cosupport of $\x_0$.
This observation suggests that we can hope to find a condition for success if we can
find some relation between $\Oper_{\cosupp^c} \hat \x_0$ and
$\Oper_{\cosupp} \hat \x_0$. 
The following result provides such a relation:

\begin{lemma}
\label{thm:GAPCoefRel}
Let $\OMeas^T$ be any $\sdim \times (\sdim-\mdim)$ basis matrix for the null space $\Null(\Meas)$ and $\cosupp$ be a fixed cosupport such that the $\cosparsity \times (\sdim-\mdim)$ matrix $\Oper_{\cosupp}\OMeas^{T}$ is of full rank $\sdim-\mdim$.
Let a signal $\x_0$ with $\Oper_\cosupp \x_0 = 0$ and its
observation $\y=\Meas\x_0$ be given. Then the estimate $\hat \x_0$ in \eqref{eq:AlgoInit} satisfies
\begin{equation}
\label{eq:CoefRel}
\Oper_\cosupp \hat\x_0 = 
-(\OMeas\Oper_\cosupp^T)^{\dagger}  \OMeas \Oper_{\cosupp^c}^T \Oper_{\cosupp^c} \hat \x_0.
\end{equation}
\end{lemma}

Having obtained a relation between $\Oper_\cosupp \hat \x_0$ and
$\Oper_{\cosupp^c} \hat \x_0$, we can derive a sufficient condition which
guarantees the success of GAP for recovering the true target signal $\x_0$:

\begin{thm}
\label{thm:GAPrecovery}
Let $\OMeas^T$ be any $\sdim \times (\sdim-\mdim)$ basis matrix for the null space $\Null(\Meas)$ and $\cosupp$ be a fixed cosupport such that the $\cosparsity \times (\sdim-\mdim)$ matrix $\Oper_{\cosupp}\OMeas^{T}$ is of full rank $\sdim-\mdim$.
Let a signal $\x_0$ with $\Oper_\cosupp \x_0 = 0$ and an
observation $\y=\Meas\x_0$ be given. Suppose that the analysis ERC~\eqref{eq:AL1ERC} holds true.
Then, when applied to solve \eqref{eq:AnalInpainting1}, 
GAP with selections factor $t \geq \opnorm{\infty\to\infty}{(\OMeas\Oper_{\cosupp}^T)^{\dagger} \OMeas \Oper_{\cosupp^c}^T }$ will recover $\x_0$ after at most $|\cosupp^c|$ iterations. 
\end{thm}

\begin{proof}
At the first iteration, GAP is doing the correct thing if it removes a row from
$\Oper_{\cosupp^c}$. Clearly, this happens when
\begin{equation}\label{eq:GAPOnestep}
\norm{\infty}{\Oper_\cosupp \hat\x_0} < t \norm{\infty}{\Oper_{\cosupp^c} \hat \x_0}.
\end{equation}
In view of \eqref{eq:CoefRel}, if \eqref{eq:AL1ERC} holds and $t \geq \opnorm{\infty\to\infty}{(\OMeas\Oper_{\cosupp}^T)^{\dagger} \OMeas \Oper_{\cosupp^c}^T }$, then
\eqref{eq:GAPOnestep} is guaranteed. Therefore, GAP successfully removes a row
from $\Oper_{\cosupp^c}$ at the first step.

Now suppose that \eqref{eq:AL1ERC} was true and GAP has removed a row
from $\Oper_{\cosupp^c}$ at the first iteration. 
Then, at the next iteration, we have the same $\Oper_{\cosupp}$ and, in the place of
$\Oper_{\cosupp^c}$,
a submatrix $\tilde \Oper_{\cosupp^c}$ of $\Oper_{\cosupp^c}$ (with one fewer row).
Thus, we can invoke Lemma~\ref{thm:GAPCoefRel} again and we have
\[
\Oper_\cosupp \hat\x_1 = - \left(\OMeas \Oper_\cosupp^T\right)^{\dagger} \OMeas \tilde \Oper_{\cosupp^c}^T \tilde \Oper_{\cosupp^c} \hat \x_1.
\]
Let $\mathbf{R}_0 := \left(\OMeas \Oper_\cosupp^T\right)^{\dagger} \OMeas \Oper_{\cosupp^c}^T$ and
$\mathbf{R}_1 := \left(\OMeas \Oper_\cosupp^T\right)^{\dagger} \OMeas \tilde \Oper_{\cosupp^c}^T$.
We observe that $\mathbf{R}_1$ 
is a submatrix of $\mathbf{R}_0$ obtained by removing one column. Therefore,
\[
\opnorm{\infty\to\infty}{\mathbf{R}_1} < \opnorm{\infty\to\infty}{\mathbf{R}_0} \leq t.
\]
By the same logic as for the first step, the success of the second step is guaranteed.
Repeating the same argument, we obtain the conclusion.
\end{proof}

\begin{remark}
As pointed out at the beginning of the subsection, the Exact Recovery Condition~\eqref{eq:AL1ERC} for the cosparse signal
recovery guarantees the success of both the GAP and the analysis $\ell_1$-minimization.
\end{remark}

\subsection{Analysis {\em vs} synthesis exact recovery conditions}
\label{sec:comparisonERC}

When $\mathbf{\Phi}$ is written as $\Meas\Dict$, the exact recovery condition~\eqref{eq:synthesisERC}
for the sparse synthesis model is equivalent to
\begin{equation}
\label{eq:OMPRecCon}
\opnorm{1\to1}{(\Meas \Dict_{\supp})^{\dagger}\Meas \Dict_{\supp^c}} < 1.
\end{equation}
Here, $\supp$ is the support of the sparsest representation of the target signal.
At first glance, the two conditions~\eqref{eq:OMPRecCon} and \eqref{eq:AL1ERC}: \[
\opnorm{1\to1}{\Oper_{\cosupp^c} \OMeas^T (\Oper_{\cosupp}\OMeas^{T})^{\dagger}}
              < 1
\]
look similar; that is, for both cases, one needs to understand the characteristics of a single matrix,
$\Oper\OMeas^T$ for the cosparse model, and $\Meas\Dict$ for the sparse model. Moreover,
the expressions involving these matrices have similar forms.

However, upon closer inspection, there is a crucial difference in the structures 
of the two expressions.
In the synthesis case, the operator norm in question depends only on how
the {\em columns} of $\Meas\Dict$ are related, since a more explicit writing of the pseudo-inverse shows that the matrix to consider is
\[
(\Dict_{\supp}^T \Meas^T \Meas \Dict_{\supp})^{-1} (\Meas \Dict_{\supp})^{T} \Meas \Dict_{\supp^c}  
\]
This fact allows us to obtain more easily characterizable conditions like
incoherence assumptions \cite{Tropp04greedis} that ensure condition~\eqref{eq:OMPRecCon}.

To the contrary, in the analysis case, more complicated relations among \emph{the rows
and the columns} of $\Oper\OMeas^T$ have to be taken into account. The matrix to consider being
\[
\Oper_{\cosupp^c}\OMeas^{T}
 \left(\OMeas \Oper_\cosupp^T \Oper_\cosupp \OMeas^T\right)^{-1} 
\OMeas \Oper_\cosupp^T,
\]
the inner expression $\OMeas \Oper_\cosupp^T \Oper_\cosupp \OMeas^T$
is connected with how the {\em columns} of $\Oper\OMeas^T$ are related. 
However, because the matrices $\Oper_{\cosupp^{c}} \OMeas^T$ and $\OMeas \Oper_{\cosupp}^T$ appear outside, it also becomes relevant how the {\em rows} of $\Oper\OMeas^T$ are related.

There is also an interesting distinction in terms of the sharpness of these exact recovery conditions. 
Namely, the violation of \eqref{eq:OMPRecCon} implies the failure of the OMP in the sense that there exist a sparse vector $\x = \Dict_{\supp} \z_{\supp}$ for which the first step of OMP picks up an atom which is not indexed by $
\supp$. To the opposite, the violation of~\eqref{eq:AL1ERC} does not seem to imply the necessary ``failure'' of GAP in a similar sense.

Note however that both conditions are not essential for the success of the algorithms.
One of the reasons is that the violation of the conditions does not guarantee that
the algorithms would select wrong atoms.
Furthermore, even if the GAP or the OMP ``fails'' in one step, that does not necessarily mean that the algorithms fail in the end: further steps may still enable them to achieve an accurate estimate of the vector $\x_{0}$.

\subsection{Relation to the Work by Cand{\`e}s et. al. \cite{Candes:2010ab}}
Before moving onto experimental results, we discuss the recovery guarantee result of Cand\`es et al.~\cite{Candes:2010ab} for the algorithm
\begin{equation}
\label{eq:noisyL1}
\hat\x = \argmin_{\hat\x\in\RR^{\sdim}} \norm{1}{\Dict^T\hat\x} \subjectto \norm{2}{\Meas\hat\x - \y} \le \epsilon
\end{equation}
when partial noisy observation $\y = \Meas\x + \w$ with $\norm{2}{\w} \le \epsilon$
is given for an unknown target signal $\x$.

In order to derive the result, the concept of D-RIP is introduced~\cite{Candes:2010ab}:
A measurement matrix $\Meas$ satisfies D-RIP adapted to $\Dict$ with constant $\delta_s^D$ if
\[
(1-\delta_s^D) \norm{2}{v}^2 \le \norm{2}{\Meas v}^2 \le (1 + \delta_s^D) \norm{2}{v}^2
\]
holds for all $v$ that can be expressed as a linear combination of $s$ columns of $\Dict$.
With this definition of D-RIP, the main result of~\cite{Candes:2010ab} can be stated as follows:
For an arbitrary tight frame $\Dict$ and a measurement matrix $\Meas$ satisfying D-RIP with $\delta_{7s}^D < 0.6$, the solution $\hat\x$ to \eqref{eq:noisyL1} satisfies
\begin{equation}
\label{eq:sparsityBasedGuarantee}
\norm{2}{\hat\x - \x} \le C_0 \epsilon + C_1 \frac{\norm{1}{\Dict^T\x - (\Dict^T \x )_s}}{\sqrt{s}}
\end{equation}
where the constants $C_0$ and $C_1$ may depend only on $\delta_{7s}^D$, and
the notation $(c)_s$ represents a sequence obtained from a sequence $c$ by keeping the $s$-largest
values of $c$ in magnitude (and setting the other to zero).

The above recovery guarantee is one of the few---very likely the only---results
existing in the literature on \eqref{eq:noisyL1}. However, we observe that there is much room for improving the result.
We now discuss why we hold this view. For clarity and for the purpose of comparison to our
result, we consider only the case $\epsilon = 0$ for \eqref{eq:noisyL1}.

First, we note that~\cite{Candes:2010ab} implicitly uses 
the estimate of type $\norm{1}{\Oper_{\Lambda^c}\z} < \norm{1}{\Oper_\Lambda \z}$
for \eqref{eq:L1NSCsimple}. 
Hence, the main result of~\cite{Candes:2010ab} cannot be sharp in general due to the fact that the sign patterns of \eqref{eq:L1NSCsimple} are ignored\footnote{Note that the same lack of sharpness holds true for our results based on~\eqref{eq:AL1ERC}, yet we will see  that these can actually provide cosparse signal recovery guarantees in simple but nontrivial cases.}

Second, the quality of the bound $\norm{1}{\Dict^T\x - (\Dict^T \x )_s}/\sqrt{s}$ 
in \eqref{eq:sparsityBasedGuarantee} is measured in terms of how effective $\Dict^T \x$ is
in sparsifying the signal $\x$ with respect to the dictionary $\Dict$.
To explain, let us consider the synthesis $\ell_1$-minimization
\begin{equation}
\label{eq:synthesisL1expanded}
\Delta_1(\x) := \argmin_{\z\in\RR^{\ddim}} \norm{1}{\z} \subjectto \Meas\Dict\z = \Meas\x
\end{equation}
and let $\Delta_0(\x)$ be the sparsest representation of $\x$.
Applying the standard result for the synthesis $\ell_1$-minimization, we have
\[
\norm{2}{\Delta_1(\x) - \Delta_0(\x)} \le C_2 \frac{\norm{1}{\Delta_0(\x) - (\Delta_0(\x))_s}}{\sqrt{s}}
\]
provided that $\Meas\Dict$ satisfies the standard RIP with, e.g., $\delta_{2s} < \sqrt{2}-1 \approx 0.414$.
Since $\Dict$ is a tight frame, it is equivalent to
\begin{equation}
\label{eq:synthesisGuarantee}
\norm{2}{\Dict\Delta_1(\x) - \x} \le C_2 \frac{\norm{1}{\Delta_0(\x) - (\Delta_0(\x))_s}}{\sqrt{s}}.
\end{equation}
Note that both $\Delta_0(\x)$ and $\Dict^T \x$ are legitimate representations of $\x$
since $\Dict\Delta_0(\x) = \x = \Dict\Dict^T\x$.
Thus, $\Delta_0(\x)$ is sparser than $\Dict^T\x$ in general; in this sense, $\Dict^T\x$ is not effective in sparsifying $\x$.
Given this, we expect that $\norm{1}{\Delta_0(\x) - (\Delta_0(\x))_s}/\sqrt{s}$
is smaller than $\norm{1}{\Dict^T\x - (\Dict^T \x )_s}/\sqrt{s}$.
We now see that \eqref{eq:sparsityBasedGuarantee} with $\epsilon=0$ and \eqref{eq:synthesisGuarantee}
are of the same form. Furthermore, given the degree of restriction on the RIP constants
($\delta_{7s}^D < 0.6$ {\em vs.} $\delta_{2s} < 0.414$),
we can only expect that the constant $C_2$ is smaller than $C_1$.
From these considerations, \eqref{eq:sparsityBasedGuarantee} only lets us to conclude that
analysis $\ell_1$-minimization \eqref{eq:AnalysisL1} performs on par with synthesis $\ell_1$-minimization \eqref{eq:synthesisL1expanded}, or tends to perform worse.

Third, the nature of the formulation in \eqref{eq:sparsityBasedGuarantee} takes the view that
the cosparse signals are the same as the sparse synthesis signals as described 
in Section~\ref{sec:ComparisonTraditionalSparsity}.
Due to this, the only way for \eqref{eq:sparsityBasedGuarantee} to explain that the cosparse signals are
perfectly recovered by analysis $\ell_1$-minimization is to show that $\Dict^T \x$
is exactly $s$-sparse for some $s>0$ with D-RIP constant $\delta_{7s}^D < 0.6$.
Unfortunately, we can quickly observe that the situation becomes hopeless even for moderately
overcomplete $\Dict$; for example, let $\Dict$ be a $1.15$-times overcomplete random tight frame
for $\RR^{\sdim}$ and consider recovering $(\sdim-1)$-cosparse signals for the operator $\Dict^T$.
Note that $(\sdim-1)$-cosparse signals $\x$ lead to $(0.15\sdim + 1)$-sparse representation
$\Dict^T\x$. This means that we need $\delta_{7(0.15\sdim+1)}^D = \delta_{1.05\sdim+7}^D$ to be
smaller than $0.6$ to show that $\x$ can be recovered with analysis $\ell_1$, which of course
cannot happen since $\delta_\sdim^D \ge 1$.
By taking the synthesis view of the signals, \eqref{eq:sparsityBasedGuarantee} cannot explain
the recovery of the simplest cosparse signals (cosparsity $\sdim-1$) no matter what $\Meas$
is (as long as it is under-determined).

We also observe that the result of \cite{Candes:2010ab} cannot say much about the recovery of
cosparse signals with respect to the finite difference operators $\Oper_{\mathrm{DIF}}$ 
discussed in Section~\ref{sec:uniqueness}. 
This is due to the fact that $\Oper_{\mathrm{DIF}}^T$ is not a tight frame. 
How does our recovery result \eqref{eq:AL1ERC} fare in this regard?
For illustration, we took $\Omega$ to be the finite difference operator $\Oper_{\mathrm{DIF}}$ 
for $32\times32$ images (thus, $\sdim=1024$). As a test image, we took $\x$ to be
constant in the region $\{ (i,j): i,j=1,\ldots,16 \}$ and $\{ (i,j): i,j=1,\ldots,16 \}^c$.
For this admittedly simple test image, we computed the operator norm
in \eqref{eq:AL1ERC} for random measurement matrices $\Meas \in \RR^{640\times1024}$.
When the operator norm was computed for $100$ instances $\Meas$, it was observed to be less than
$0.726$. Hence, our result does give the guarantee of cosparse signal recovery in simple cases.

\section{Experiments}
\label{sec:experiment}

Empirical performance of the proposed algorithms is presented in this section.
First, we show how the algorithms perform in synthetic cosparse recovery problems.
Second, 
experimental results for an analysis-based compressed sensing are presented.

\subsection{Performance of analysis algorithms}
\label{sec:performanceAnalysis}
In this section, we apply the algorithms described in Section~\ref{sec:algorithm} to synthetic
cosparse recovery problems. 
In the experiment, the entries of $\Meas \in \RR^{\mdim\times \sdim}$ were drawn
independently from the normal distribution.
For the analysis operator $\Oper \in \RR^{\pdim\times\sdim}$, it was constructed so that its transpose
is a random tight frame with unit norm columns---we will simply say that $\Oper$ is a random tight
frame in this case.\footnote{One could also construct $\Oper$ by simply drawing the rows of it randomly and independently from
$\Sphe^{\sdim-1}$ without the tight frame constraint. We have run the experiment for such operators and observed that the
result was similar.}
Next, the co-sparsity $\cosparsity$ was chosen,
and the true or target signal $\x$ was generated randomly as described in Section~\ref{sec:GenModel}.
The observation was obtained by $\y = \Meas\x$.

We have used Matlab \texttt{cvx} package~\cite{cvx} with the precision set to \texttt{best}
for the analysis-$\ell_1$. For the final results, we used the estimate $\hat \x$ from $\ell_1$ solver to obtain an estimate
of the cosupport---the cosupport estimate was obtained by taking the indices for which
the corresponding analysis coefficient is of size less than $10^{-6}$---and then
using this cosupport and the observation $\y$ to compute the final estimate of $\x$
(this process can be considered as de-biasing.).

\begin{figure}[htbp]
\centering
\includegraphics[width=\textwidth*2/7]{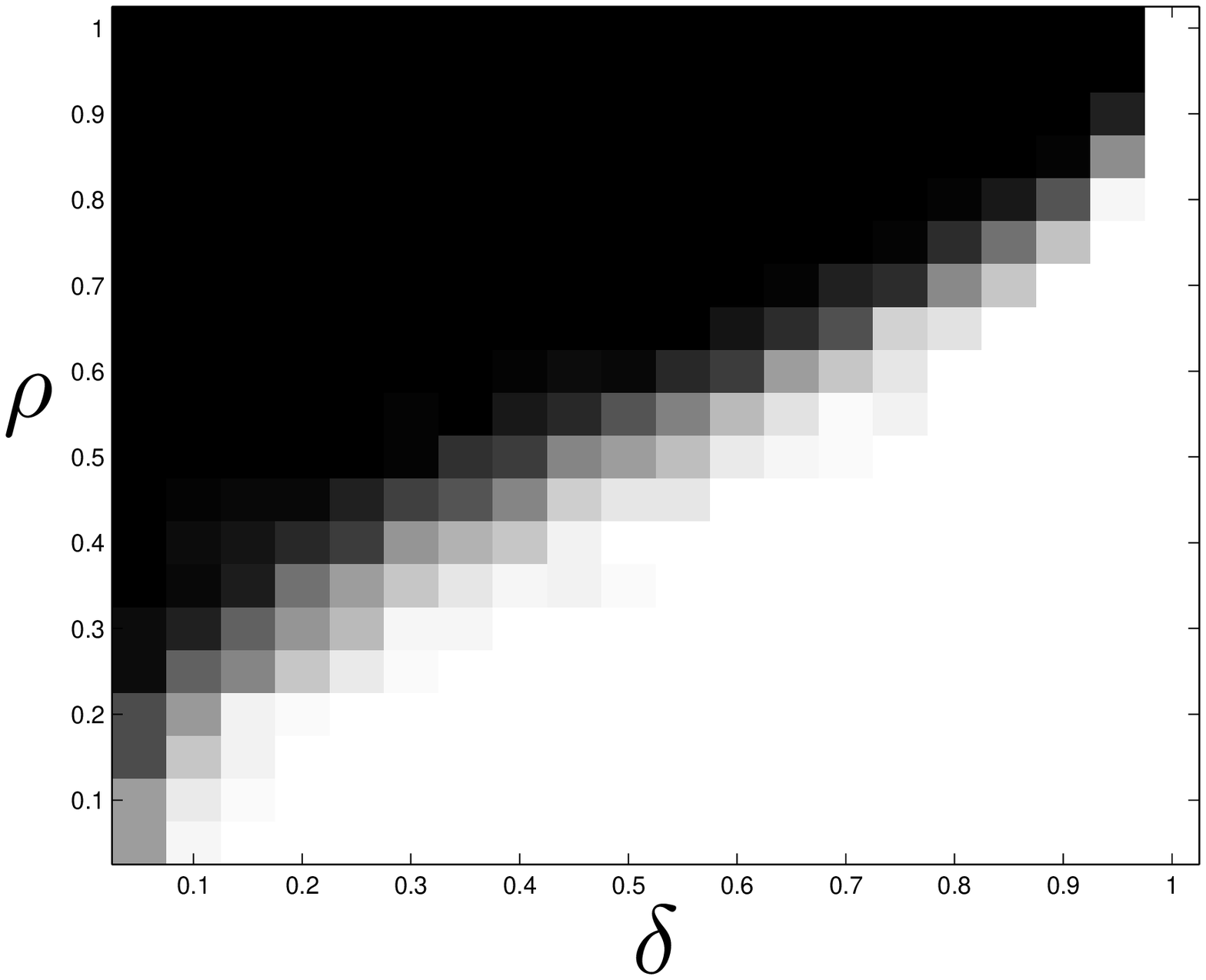}
\includegraphics[width=\textwidth*2/7]{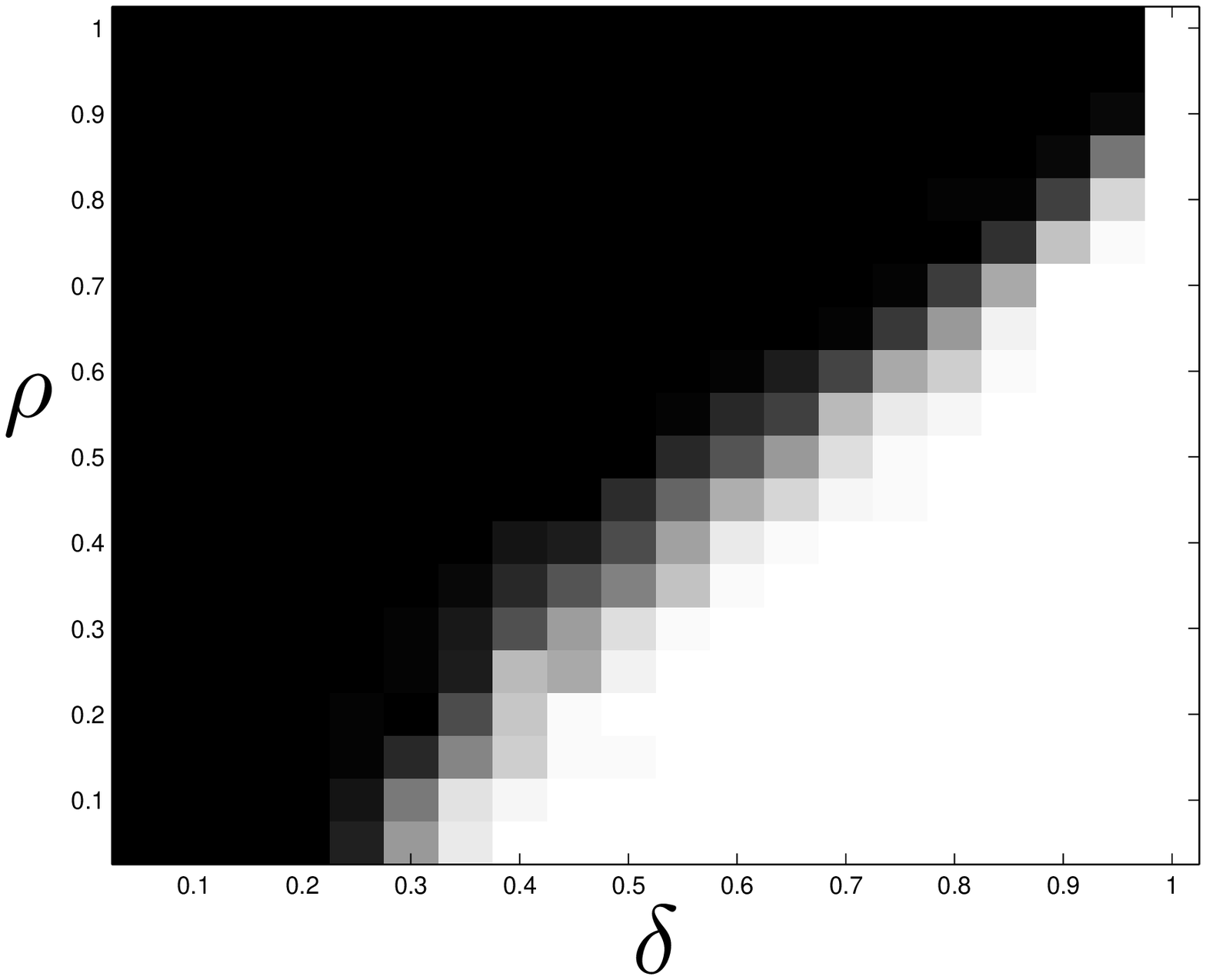}
\includegraphics[width=\textwidth*2/7]{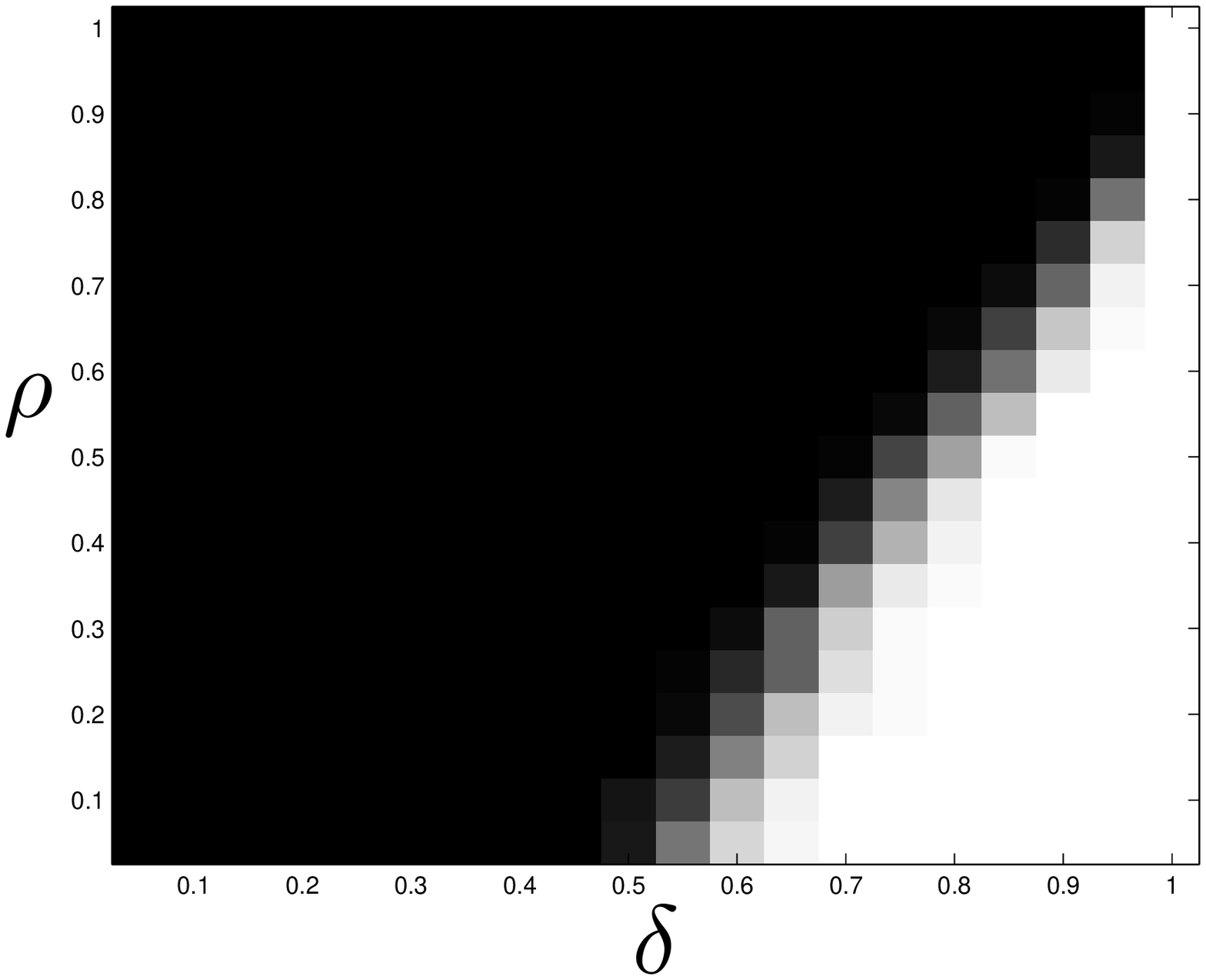}
\includegraphics[width=\textwidth*2/7]{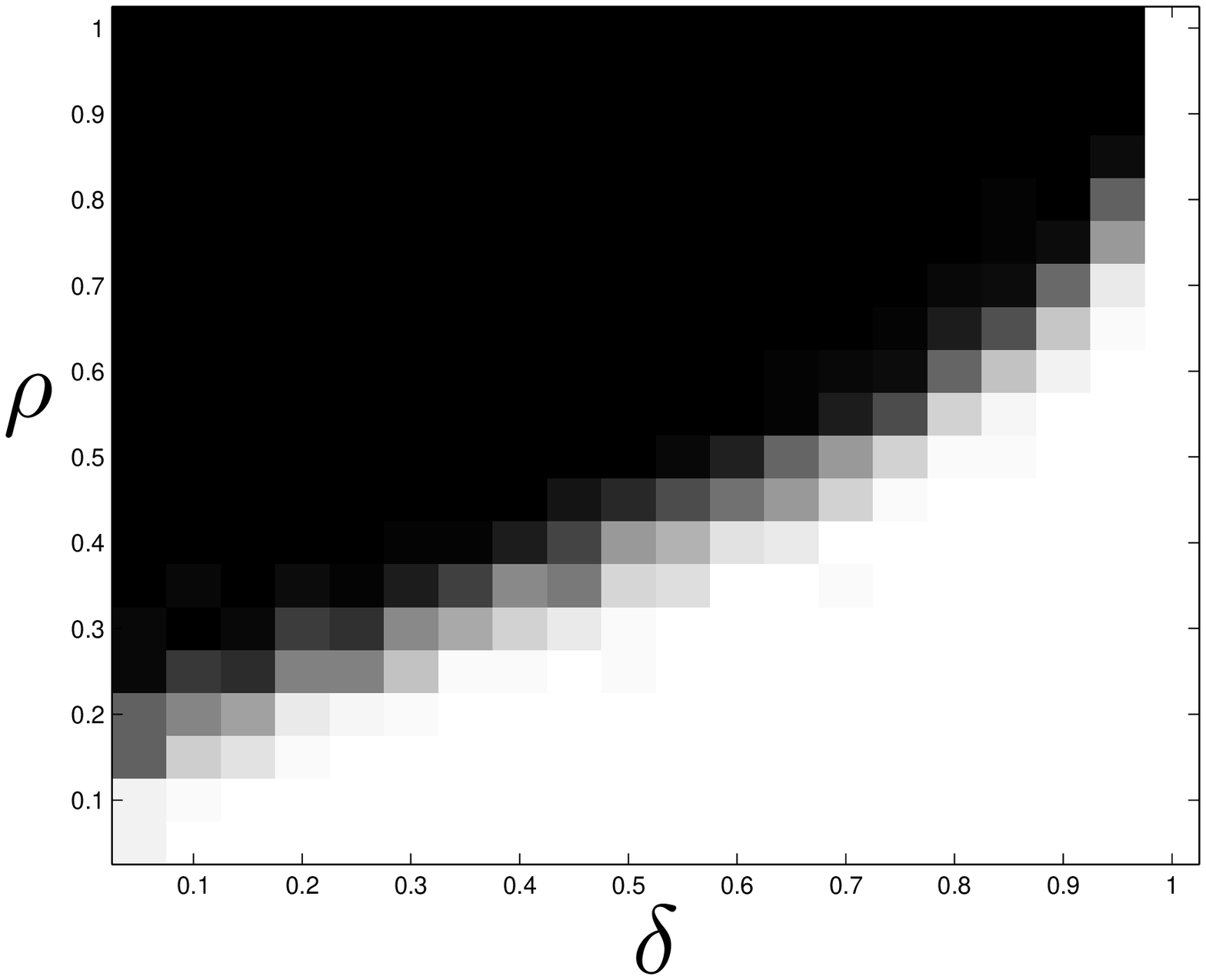}
\includegraphics[width=\textwidth*2/7]{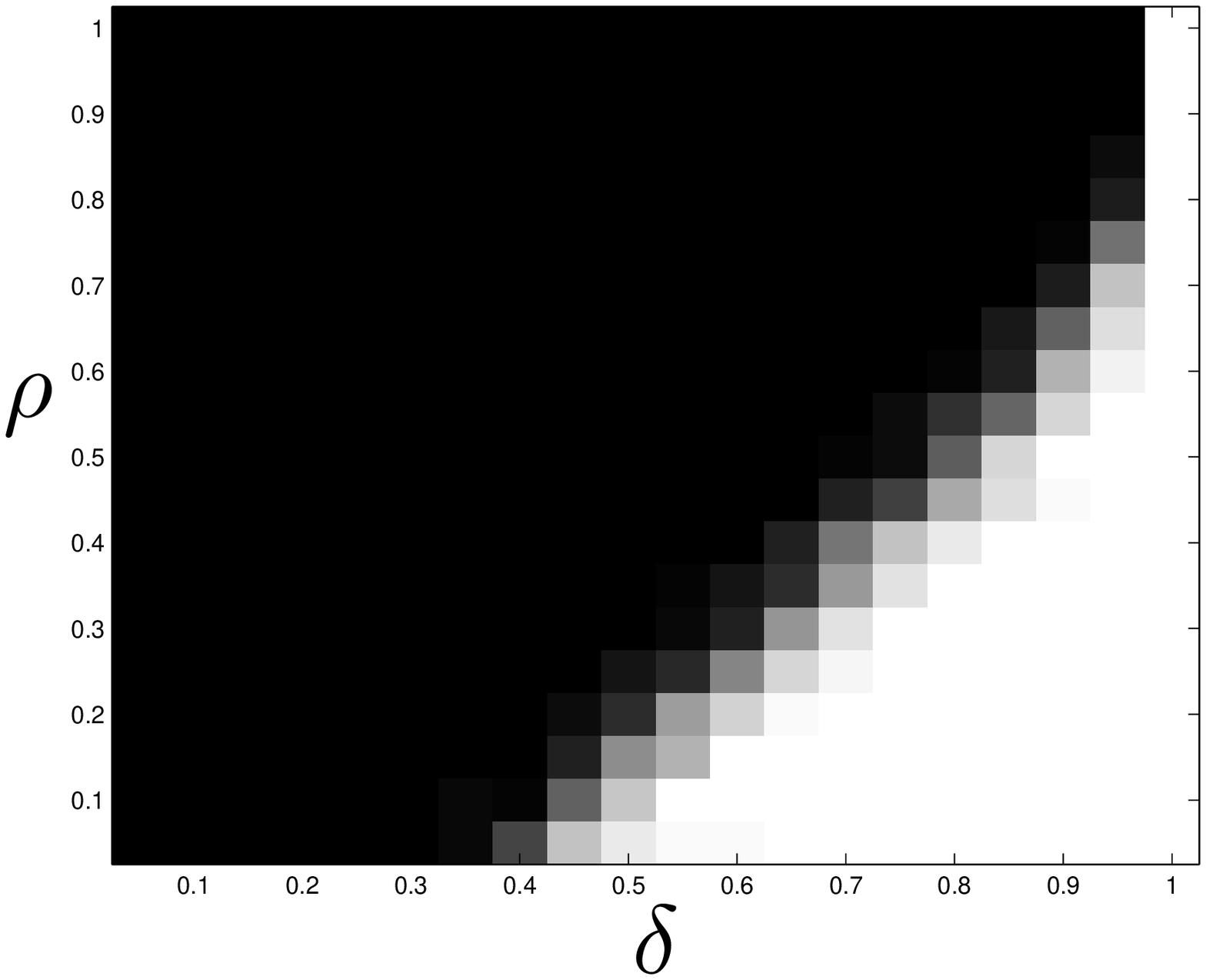}
\includegraphics[width=\textwidth*2/7]{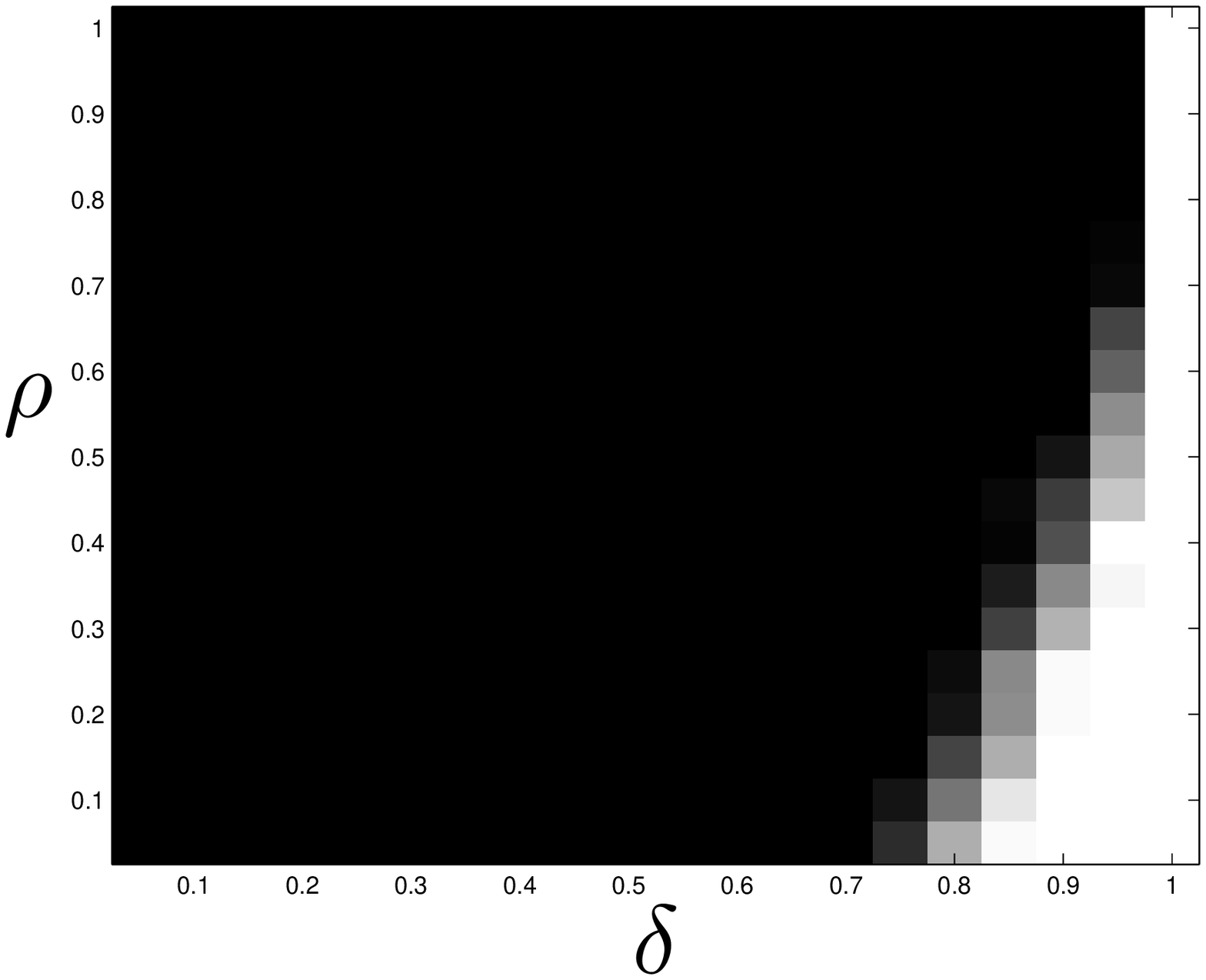}
\caption{Recovery Rate of Analysis Algorithms for $\sdim=200$. The figures
correspond to GAP (top) and L1 (bottom) with $\sigma=1$ (left), $\sigma=1.2$ (center) and $\sigma=2$ (right).}
\label{fig:perfAAlgs1}
\end{figure}

Figure \ref{fig:perfAAlgs1} 
 shows the results. In all cases, the signal dimension $\sdim$ is set to
$200$. We then varied the number $\mdim$ of measurements, the co-sparsity $\cosparsity$ of
the target signal, and the operator size $\pdim$ according to the following formulae:
\[
\mdim = \delta \sdim, \quad
\cosparsity = \sdim - \rho \mdim, \quad
\pdim = \sigma \sdim. 
\]
which is consistent with Donoho \& Tanner's notations for phase transition diagrams~\cite{Donoho:2009aa}: $\delta = \mdim/\sdim$ is the undersampling ratio, and $\rho = (\sdim-\cosparsity)/\mdim$ measures the relative dimension of the $\cosparsity$-cosparse subspaces compared to the number of measures.
For every fixed parameter triplet $(\sigma, \delta, \rho)$, 
the experiment was repeated $50$ times. A relative error of size less than $10^{-6}$
was counted as perfect recovery.
Each pixel in the diagrams corresponds to a triplet $(\sigma, \delta, \rho)$
and the pixel intensity represents the ratio of the signals recovered perfectly with
white being the 100\% success.

The figures show that the GAP can be a viable option when it comes to the cosparse
signal recovery.
What is a bit unexpected is that
GAP performs better than $\ell_1$-minimization, especially for overcomplete $\Oper$'s. Yet, it should be clear from its description that GAP has polynomial complexity, and it is tractable in practice.

An interesting phenomenon observed in the plots for overcomplete $\Oper$ is that there seems to
be some threshold $\delta_*$ such that if the observation to dimension ratio $\delta$ is less
than $\delta_*$, one could not recover any signal however cosparse it may be.
We may explain this heuristically as follows:
If $\mdim$ measurements are available, then the amount of information we have for the signal
is $c_1 \mdim$ where $c_1$ is the number of bits each observation represent.
In order to recover a cosparse signal, we need first to identify which subspace the signal belongs
to out of $\binom{\pdim}{\cosparsity}$, and then to obtain the $\sdim-\cosparsity$ coefficients
for the signal with respect to a basis of the $\sdim-\cosparsity$ dimensional subspace.
Therefore, roughly speaking, one may hope to recover the signal when
\[
c_1 \mdim \geq \log_2 \binom{\pdim}{\cosparsity} + c_1(\sdim-\cosparsity) = \log_2 \binom{\pdim}{\cosparsity} + \rho c_1\mdim.
\]
Thus, the recovery is only possible when $(1-\rho)\delta \geq \log_2 \binom{\pdim}{\sdim} / (c_1 \sdim)$.
Using the relation $\pdim=\sigma\sdim$ and Stirling's approximation, this leads to an asymptotic
relation
\[
\delta \geq (1-\rho)\delta \geq \frac{\sigma\log\sigma - (\sigma-1) \log(\sigma-1)}{c_1},
\]
which explains the phenomenon.

The calculation above and the experimental evidence from the figures confirm the intuition we had
in Section~\ref{sec:Subspaces}: The combinatorial number of low-dimensional cosparse subspaces 
arising from analysis operators in general position is not desirable.
This strengthens our view on the necessity of designing/learning analysis operators with high
linear dependencies.

\subsection{Analysis-based Compressed Sensing}
\label{sec:csrecovery}
We observed in Section~\ref{sec:performanceAnalysis} that the cosparse analysis model 
facilitates effective algorithms to recover partially observed cosparse signals. 
In this section, 
we demonstrate the effectiveness of GAP algorithm on a standard toy problem: the Shepp Logan phantom recovery problem.

We consider the following problem that is related to computed tomography (CT):
There is an image, say of size $n\times n$, which we are interested in but cannot observe directly.
It can only be observed indirectly by means of its 2D Fourier transform coefficients.
However, due to high cost of measurements or some physical limitation,
the Fourier coefficients can only be observed along a few radial lines.
These limited observations or the locations thereof can be modeled by a 
measurement matrix $\Meas$, and with the obtained observation we want to recover
the original image. As an ideal example, we consider the Shepp Logan phantom.
One can easily see that this image is a good example of cosparse signals
in $\Oper_{\mathrm{DIF}}$ which consists of all the vertical and horizontal gradients
(or one step differences).
This image has been used extensively as an example in the literature in the context of
compressed sensing (see, e.g., \cite{CandRombTao:2006b,blumensath:aiht}).

Figure~\ref{fig:losh256} is the result obtained using GAP. The number
of measurements that corresponds to $12$ radial lines is $\mdim = 3032$. Compared to
the number of pixels in the image $\sdim = 65536$, it is approximately $4.63$\%. The number of analysis atoms
that give non-zero coefficients is $\pdim-\cosparsity = 2546$. The size of $\Oper_{\mathrm{DIF}}$ is 
roughly twice the image size $\sdim=65536$, namely $\pdim=130560$. 
At first glance, this corresponds to
very high co-sparsity level ($\cosparsity=130560-2546$), or put differently,
given the high cosparsity level $\cosparsity=128014$, we seem to have required too many measurements.
However, using the conjectured near optimal necessary condition for
uniqueness guarantee~\eqref{eq:DimMaxBoundsDIF},
we may have uniqueness guarantee when $\mdim \ge 2551$. 
Also, using the sufficient condition~\eqref{eq:scUniqueness2DTV}, one would want to have
$\mdim \ge 3058$ measurements.
In view of this, the fact that GAP recovered the signal perfectly for $3032$ measurements
is remarkable!

\begin{figure}[htbp]
\begin{center}
\includegraphics[width=\textwidth*3/10]{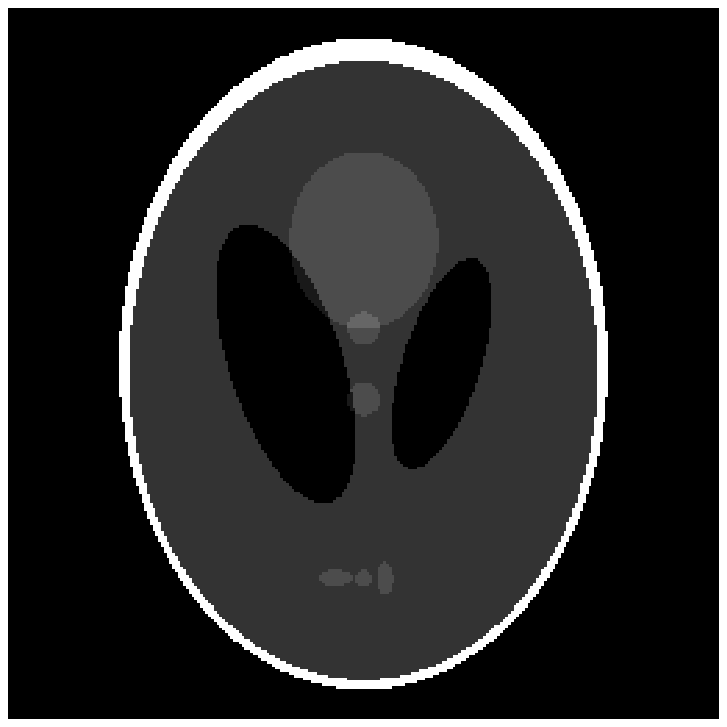}
\includegraphics[width=\textwidth*3/10]{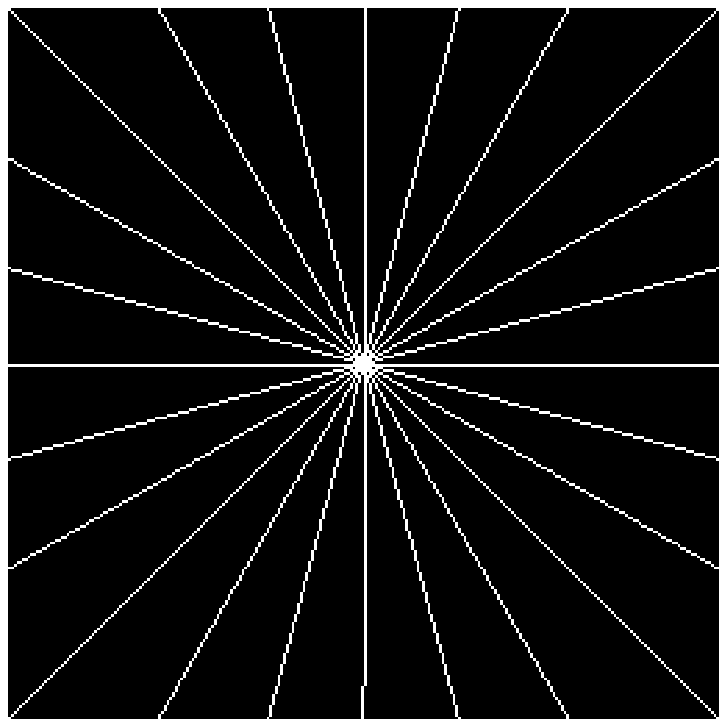}
\includegraphics[width=\textwidth*3/10]{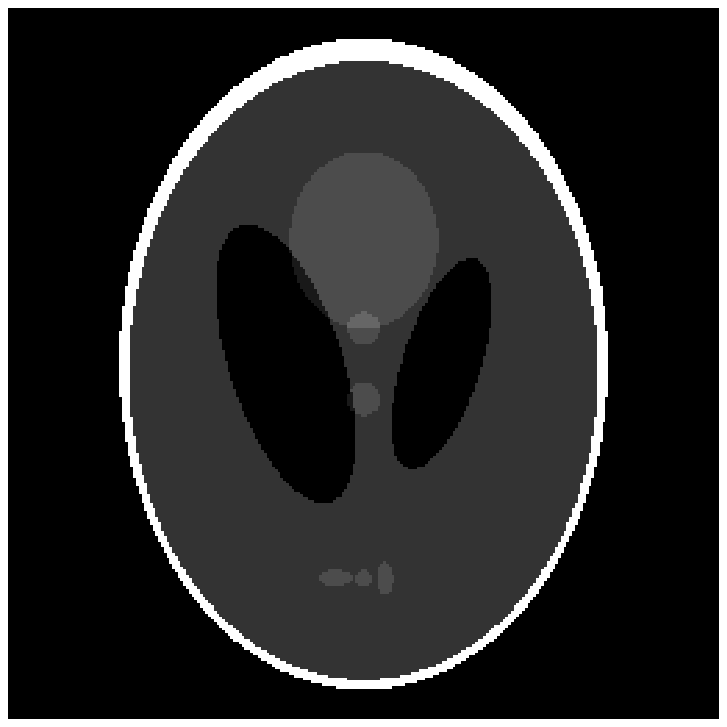}
\includegraphics[width=\textwidth*3/10]{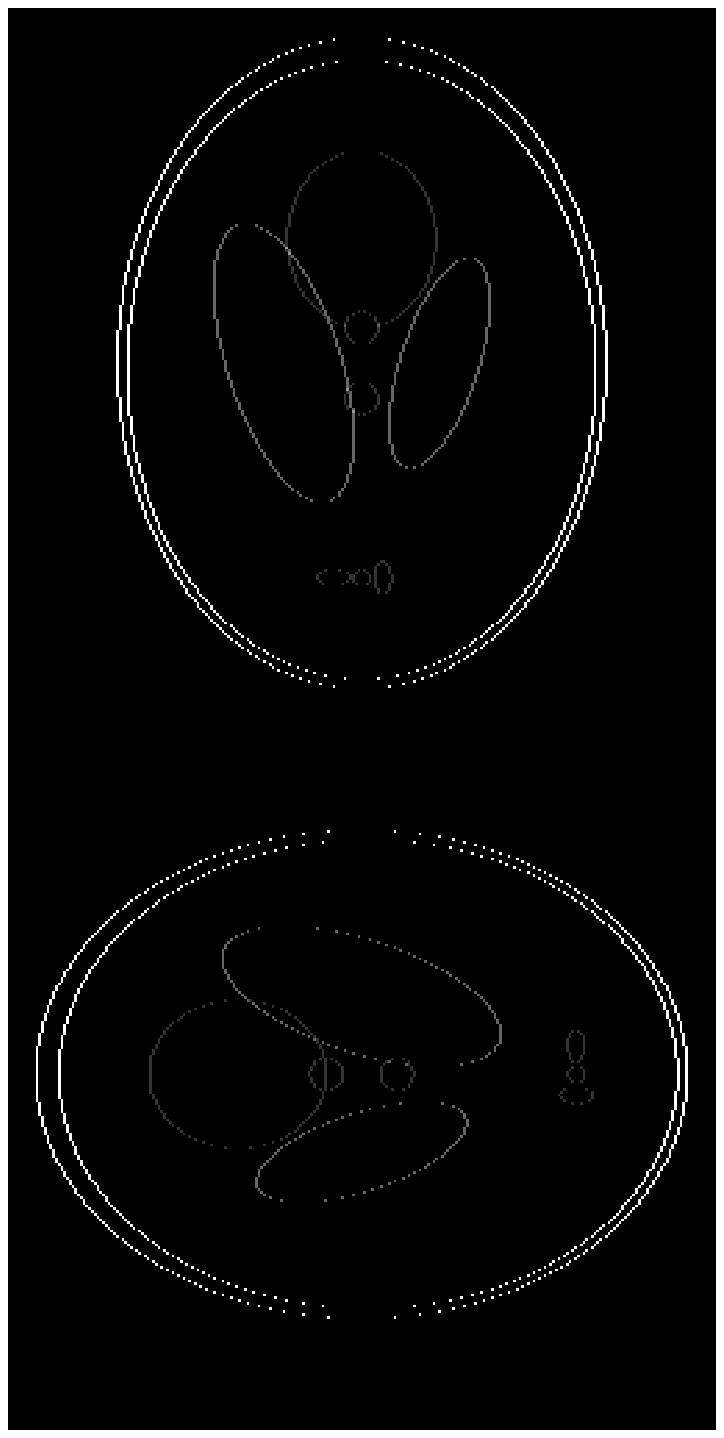}
\includegraphics[width=\textwidth*3/10]{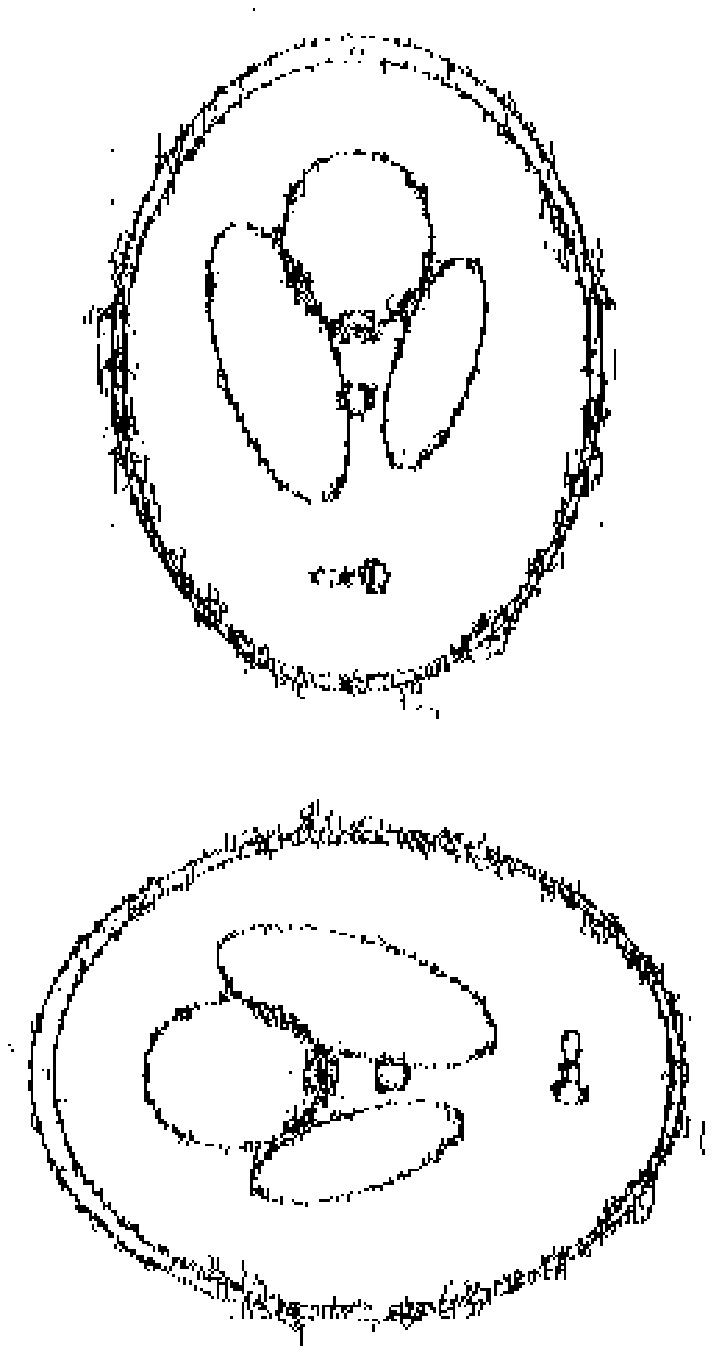}
\includegraphics[width=\textwidth*3/10]{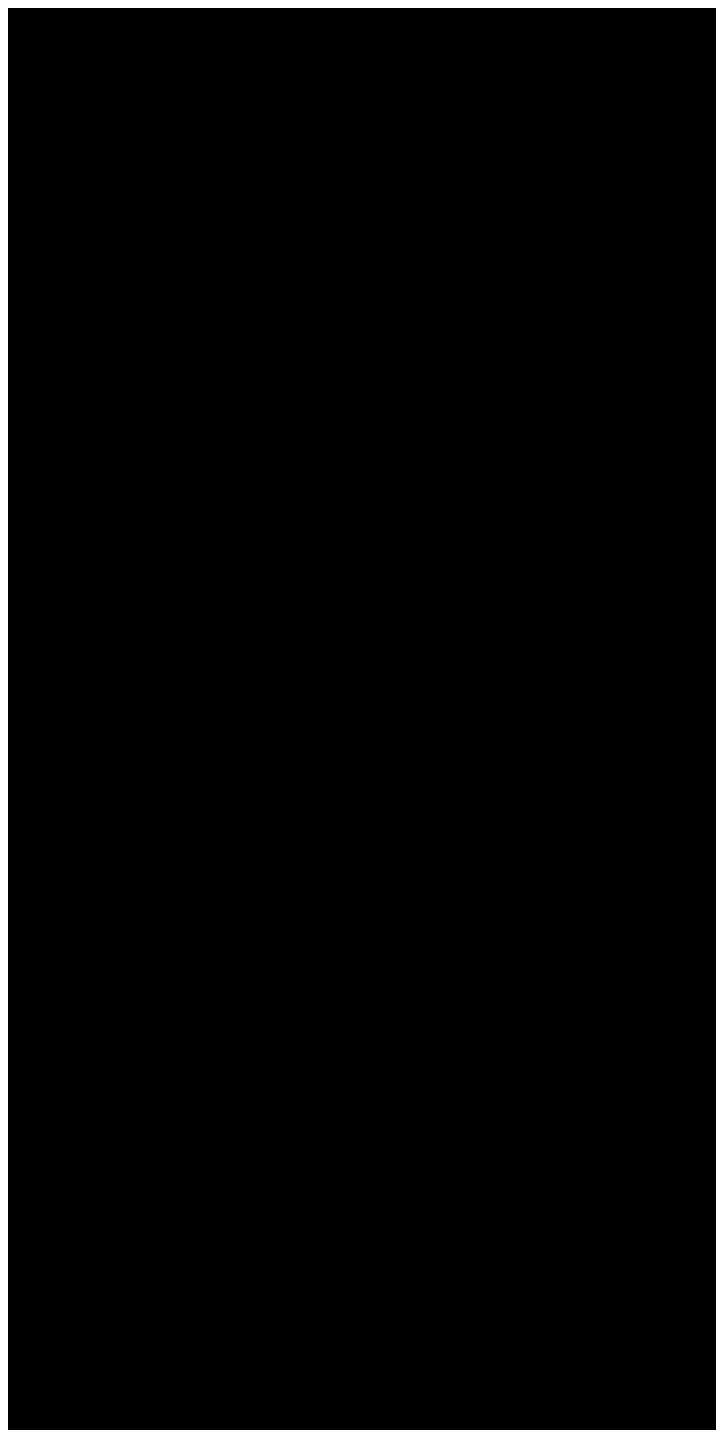}
\caption{Recovery of $256\times256$ Shepp Logan phantom image.
From top to bottom, left to right: (a) Original Image. (b) Sampling locations
of Fourier coefficients. 
(c) Reconstructed image.
(d) Locations where one-step difference of the original
image is non-zero. Upper half corresponds to the horizontal differences and
lower half the vertical differences. (e) Locations that GAP identified/eliminated
to be the ones where the differences are likely non-zero. (f) Locations that
GAP failed to identify as non-zero locations. Blank black figure indicates
that none of the non-zero locations were missed (perfect reconstruction). 
}
\label{fig:losh256}
\end{center}
\end{figure}

We have also ran the GAP algorithm for a larger sized $512\times 512$ problem. The results (not shown here) are visually similar to Figure~\ref{fig:losh256}.
In this case, the number of measurements
($\mdim = 7112$) represents approximately $2.71$\% of the image size ($\sdim = 262144$).
The number of non-zero analysis coefficients is $\pdim-\cosparsity = 5104$.
The sufficient uniqueness condition \eqref{eq:scUniqueness2DTV} gives
$\mdim \geq 6126$ as a number of measurements for the uniqueness.

\begin{remark}
Due to the large size of these problems, GAP algorithm as described in Section~\ref{sec:algorithm} had to be modified: 
We used numerical optimization to compute pseudo-inverses. 
Also, due to high computational cost, we eliminated many rows at each iteration (super
greedy) instead of one. Although this was not implemented using a selection factor, this can be interpreted as using varying selection factors $0<t_{k}<1$ along the iterations. 
\end{remark}

To conclude this section, we have repeated the $256\times256$ Shepp Logan phantom image recovery
problem for several algorithms while varying the number of radial observation lines.
Given that we know the minimal theoretical number and a theoretically sufficient number
of radial observation lines for the uniqueness guarantee, the experimental result gives
us an insight on how various algorithms actually perform in the recovery problem in relation
to the amount of observation available. Figure~\ref{fig:varyingMeasurements} shows the outcome.
The algorithms used in the experiment are the GAP, the TV-minimization from \texttt{l1magic},
the AIHT from \cite{blumensath:aiht}, and the back-projection algorithm.\footnote{The code for \texttt{l1magic} was downloaded from \texttt{http://www.acm.caltech.edu/l1magic/} and the one for AIHT
from \texttt{http://www.personal.soton.ac.uk/tb1m08/sparsify/AIHT\_Paper\_Code.zip}. The result
for the back-projection was obtained using the code for AIHT.}
The GAP and \texttt{l1magic} can be viewed as analysis-based reconstruction algorithms while
the AIHT is a synthesis-based reconstruction algorithm.
The AIHT is seen to use Haar wavelets as the synthesis dictionary, hence the algorithm implicitly
assumes that the phantom image has sparse representation in that dictionary.
We remark that while Figure~\ref{fig:varyingMeasurements} gives an impression that the AIHT
does not have any improvement over the baseline back-projection algorithm, perfect reconstructions
were observed for the former when sufficient measurements were available, which is not the
case for the latter.

\begin{figure}[htbp]
\centering
\includegraphics[width=\textwidth*8/10]{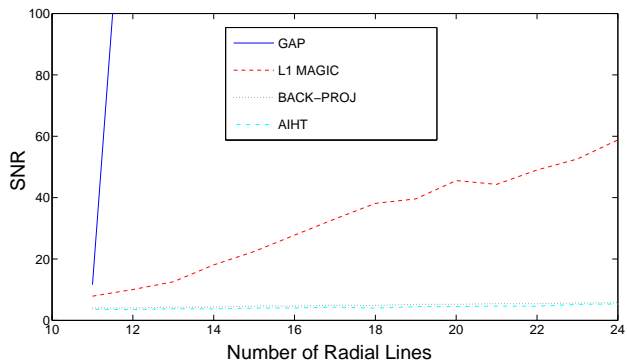}
\caption{SNR vs the number of radial observation lines in $256\times256$ Shepp Logan phantom
image recovery. The output line for the GAP is clipped due to high SNR value.}
\label{fig:varyingMeasurements}
\end{figure}

\begin{remark}
It must be noted that in our experiment, each radial line consists of $N$ pixels for an $N\times N$
image; this is in contrast to the fact that the radial lines in the existing codes, e.g. \texttt{l1magic}, have $N-1$ pixels. 
We have made appropriate changes for our experiment.
The radial lines with $N-1$ pixels do make the recovery problem more difficult and
more observations were required for perfect recovery for the GAP.
\end{remark}

\section{Conclusions and Further Work}
\label{sec:final}

In this work, we have described the cosparse analysis data model as an 
alternative to the popular sparse synthesis model.
By the description, we have shown that the cosparse analysis model is
distinctly different from the sparse synthesis one in spite of their
apparent similarities. 
In particular, treating the cosparse model as the synthesis model
by assuming that the analysis representations of cosparse signals
are sparse was demonstrated to be not very meaningful. 
Having had presented the model, we have stated conditions that guarantee
the uniqueness of cosparse solutions in the context of linear inverse
problems based on the work \cite{Lu:2008ab}.
We then presented some algorithms for the cosparse recovery problem
and provided some theoretical result for the analysis $\ell_1$-minimization
and the newly proposed GAP.
Lastly, the model and the proposed algorithm were validated via experimental
results.

Although our work in this paper shows that the cosparse analysis model
together with algorithms based on the model is an interesting subject to
study and viable for practical applications, there are much more to be
learned about the model. Among possible future avenues for related research,
we list the following:
1) The stability of measurement matrices $\Meas$ on the analysis
union of subspaces $\cup_{\cosupp} \aspace_\cosupp$;
2) The effect of noise on the cosparse analysis model and associated algorithms;
3) The designing / learning of analysis operators for classes of signals of interest;
4) More concrete and/or optimal theoretical success guarantees for
algorithms, with a better understanding of the role of linear dependencies between rows of the analysis operator.

\appendix
\section{Proof of Theorem~\ref{thm:L1NSC} and Corollary~\ref{thm:L1ERC}}
\label{sec:proofL1}

Let us begin with the simplest case.
For a fixed $\x_0$ with cosupport $\cosupp$,
the analysis $\ell_1$-minimization~\eqref{eq:analysisL1}
recovers $\x_0$ as the unique minimizer if and only if
\[
\left| \innerp{\Oper_{\cosupp^c}\z}{\sign(\Oper_{\cosupp^c}\x_0)} \right|
< \norm{1}{\Oper_\cosupp\z}, \quad \forall \z \in \Null(\Meas),\ \z\neq0.
\]
This follows from two facts: a) the above condition characterizes strict local minima of the optimization problem; b) the optimization problem is convex and can have at most one strict local minimum, which must be the unique global optimum.
From this, we derive the following:
The analysis $\ell_1$-minimization~\eqref{eq:analysisL1} recovers $\x_0$ 
as a unique minimizer for \emph{any} $\x_0$ with cosupport $\cosupp$, if and only if
\[
\sup_{\x_\cosupp : \Oper_\cosupp\x_\cosupp = 0}
\left| \innerp{\Oper_{\cosupp^c}\z}{\sign(\Oper_{\cosupp^c}\x_\cosupp)} \right|
< \norm{1}{\Oper_\cosupp\z}, \quad \forall \z \in \Null(\Meas),\ \z\neq0
\]
and the proof of Theorem~\ref{thm:L1NSC} is complete.

To obtain Corollary~\ref{thm:L1ERC}, observe that we can remove the constraint $\z\in\Null(\Meas)$ by writing
$\z = \OMeas^T\alpha$ where $\OMeas^T$ is an $\sdim \times (\sdim-\mdim)$ basis matrix for $\Null(\Meas)$ and $\alpha \in \RR^{\sdim-\mdim}$
is an appropriate coefficient sequence. Thus, the necessary and sufficient condition becomes
\begin{align}
\sup_{\x_\cosupp : \Oper_\cosupp\x_\cosupp = 0}
&\left| \innerp{\Oper_{\cosupp^c} \OMeas^T\alpha}{\sign(\Oper_{\cosupp^c}\x_\cosupp)} \right|
< \norm{1}{\Oper_\cosupp \OMeas^T \alpha},\quad \forall \alpha \in \RR^{\sdim-\mdim},\ \alpha \neq 0.
\intertext{Since the $\cosparsity \times (\sdim-\mdim)$ matrix $\Oper_\cosupp \OMeas^T$ is thin ($\cosparsity \geq \sdim-\mdim$) and full-rank, defining $\beta := \Oper_\cosupp \OMeas^T \alpha$, we have $\alpha = (\Oper_\cosupp \OMeas^T)^\dagger \beta$.
Therefore, a {\em sufficient} (but no longer necessary) recovery condition for analysis $\ell_{1}$-minimization is}
\label{eq:AERC1}
\sup_{\x_\cosupp : \Oper_\cosupp\x_\cosupp = 0}
&\left| \innerp{\Oper_{\cosupp^c} \OMeas^T (\Oper_\cosupp \OMeas^T)^\dagger \beta}
              {\sign(\Oper_{\cosupp^c}\x_\cosupp)} \right|
< \norm{1}{\beta},\quad \forall \beta \in \RR^{\cosparsity}, \beta \neq 0.
\intertext{Equivalently, for all $\x_\cosupp$ with $\Oper_\cosupp\x_\cosupp = 0$,}
\label{eq:AERC2}
\sup_{\norm{1}{\beta}=1}
&| \innerp{\beta}  {
   ( \OMeas\Oper_\cosupp^T)^\dagger \OMeas\Oper_{\cosupp^c} ^T
   \sign(\Oper_{\cosupp^c}\x_\cosupp)}|
< 1
\intertext{
that is to say
}
\label{eq:AERC3}
\sup_{\x_\cosupp : \Oper_\cosupp\x_\cosupp = 0}
&\norm{\infty}         
{
( \OMeas\Oper_\cosupp^T)^\dagger \OMeas\Oper_{\cosupp^c} ^T               
\sign(\Oper_{\cosupp^c}\x_\cosupp)}
< 1.
\end{align}
Condition~\eqref{eq:L1NSC} follows from the above.
To conclude the proof of Corollary~\ref{thm:L1ERC}, we note that since $\norm{\infty}{\sign(\Oper_{\cosupp^c}\x_\cosupp)} = 1$, the left hand side of \eqref{eq:AERC3}
is bounded above by
\[
\opnorm{\infty\to\infty}{
{( \OMeas\Oper_\cosupp^T)^\dagger \OMeas\Oper_{\cosupp^c} ^T}}
=
\opnorm{1\to1}{
{\Oper_{\cosupp^c} \OMeas^T (\Oper_\cosupp \OMeas^T)^{\dagger}}}.
\]
Therefore, condition~\eqref{eq:AL1ERC} implies \eqref{eq:L1NSC} and the proof is complete.

\section{Proof of Lemma \ref{thm:GAPCoefRel}}
\label{sec:proofGAP}

Since $\hat \x_0$ is the solution of
$\arg\min_{\x} \norm{2}{\Oper \x}^2 \text{ subject to } \y = \Meas \x$,
applying the Lagrange multiplier method, we observe that $\hat \x_0$ satisfies
\[
\Oper^T\Oper\hat \x_0 = \Meas^T \v \quad\text{and}\quad
\Meas \hat \x_0 = \y,
\]
for some $\v \in \RR^m$. From the first equation, we obtain
$\v = (\Meas^{T})^{\dagger}\Oper^{T} \Oper \hat\x_{0}$.
Putting this back in, one gets
\(
\left( \Id - \Meas^T(\Meas^{T})^{\dagger} \right) \Oper^T\Oper \hat \x_0 = 0.
\)
The last equation can be written as
$(\OMeas^T)^\dagger\OMeas \Oper^T\Oper \hat \x_0 = 0$, where $(\OMeas^T)^\dagger$ is
the pseudo-inverse of $\OMeas^T$.
Thus,
\[
\OMeas \Oper^T\Oper \hat \x_0 = 0.
\]
Now, we split
$
\Oper^T\Oper = \Oper_\cosupp^T \Oper_\cosupp + \Oper_{\cosupp^c}^T \Oper_{\cosupp^c}
$
and write
\[
\OMeas \Oper_\cosupp^T \Oper_\cosupp \hat \x_0 = - \OMeas \Oper_{\cosupp^c}^T \Oper_{\cosupp^c} \hat \x_0.
\]
Since $\Oper_\cosupp \x_0 = 0$, we can also write
\begin{equation}\label{eq:gaprec1}
\OMeas \Oper_\cosupp^T \Oper_\cosupp \u = - \OMeas \Oper_{\cosupp^c}^T \Oper_{\cosupp^c} \hat \x_0
\end{equation}
with $\u = \hat\x_0 - \x_0$. 
On the other hand, from $\Meas \hat \x_0 = \y = \Meas\x_0$, we have
\(
\Meas \u = 0.
\)
This means that $\u$ can be expressed as $\u =: \OMeas^T \w$ for some $\w$. 
Plugging this into \eqref{eq:gaprec1},
we have
\[
\OMeas \Oper_\cosupp^T \Oper_\cosupp \OMeas^T \w = - \OMeas \Oper_{\cosupp^c}^T \Oper_{\cosupp^c} \hat \x_0.
\]
Hence, 
$
\w = - \left(\OMeas \Oper_\cosupp^T \Oper_\cosupp \OMeas^T\right)^{-1} \OMeas \Oper_{\cosupp^c}^T \Oper_{\cosupp^c} \hat \x_0.
$
This gives us
\[
\hat\x_0 - \x_0 = \u = - \OMeas^T \left(\OMeas \Oper_\cosupp^T \Oper_\cosupp \OMeas^T\right)^{-1} \OMeas \Oper_{\cosupp^c}^T \Oper_{\cosupp^c} \hat \x_0.
\]
Again, using $\Oper_\cosupp \x_0 = 0$, we have
\[
\Oper_\cosupp \hat\x_0 = - \Oper_\cosupp \OMeas^T \left(\OMeas \Oper_\cosupp^T \Oper_\cosupp \OMeas^T\right)^{-1} \OMeas \Oper_{\cosupp^c}^T \Oper_{\cosupp^c} \hat \x_0 = -(\OMeas\Oper_\cosupp^T)^{\dagger}  \OMeas \Oper_{\cosupp^c}^T \Oper_{\cosupp^c} \hat \x_0.
\]

\section{Proof of Proposition \ref{thm:Uniqueness2DTV}}
\label{sec:proofsGraph}

All the statements in this section are about a 2D regular graph
consisting of $\sdim=N\times N$ vertices ($V$) and the vertical and horizontal edges ($E$) 
connecting these vertices.
To prove the proposition, we will start with two simple lemmas.

\begin{lemma}
\label{thm:oneComponent}
For a fixed $\cosparsity$, the value 
\[
\alpha(\cosparsity) := 
      \min_{\cosupp \subset E: |\cosupp| \ge \cosparsity} \{ |V(\cosupp)| - J(\cosupp) \}
\]
is achieved for a subgraph $(V(\cosupp),\cosupp)$---we will simply identify $\cosupp$ with
the subgraph from here on---satisfying $|\cosupp| = \cosparsity$ and $J(\cosupp) = 1$.
\end{lemma}

\begin{proof}
It is not difficult to check that the minimum if achieved for $\cosupp$ with $|\cosupp| = \cosparsity$. Thus, we will assume $|\cosupp| = \cosparsity$.

Now, we need to show that there is also a $\cosupp$ with $J(\cosupp)=1$.
Suppose that $\tilde\cosupp$ with $|\tilde\cosupp|=\cosparsity$ achieves $\alpha(\cosparsity)$
and $J(\tilde\cosupp) > 1$. We will show that we can obtain $\cosupp$ from $\tilde\cosupp$
that also achieves the value $\alpha(\cosparsity)$, and $|\cosupp|=\cosparsity$ and $J(\cosupp)=1$.
For simplicity, we will consider the case $J(\tilde\cosupp) = 2$ only; one can deal with other cases
by the repetition of the same argument.

Let $\tilde\cosupp_1$ and $\tilde\cosupp_2$ be the two connected components of $\tilde\cosupp$.
Note that on a 2D regular graph, we can shift a subgraph horizontally or vertically unless
the subgraph has vertices on all four boundaries of $V$. 
Since $\tilde\cosupp_1$ and $\tilde\cosupp_2$ are disconnected, not all of them can have vertices
on all four boundaries of $V$. Therefore, one of them, say $\tilde\cosupp_1$, can be shifted
towards the other. Let us consider the first moment when they touched each other.
Let $t$ be the number of vertices that coincided. Then, at most $t-1$ edges must have coincided.
Thus, denoting the number of edges coincided by $s < t$,
the resulting subgraph $\tilde\cosupp'$ has $|V(\tilde\cosupp)|-t$ vertices and
$|\tilde\cosupp|-s$ edges and one connected components.
Now let $\cosupp$ be a subgraph obtained from $\tilde\cosupp'$ by adding $s$ additional edges that
are connected to $\tilde\cosupp'$.
Then, 
\[
|V(\cosupp)| \le |V(\tilde\cosupp')| + s \le |V(\tilde\cosupp)|-t + s,
\] 
$|\cosupp| = |\tilde\cosupp| = \cosparsity$, and $J(\cosupp) = 1$.
Hence,
\[
|V(\cosupp)|-J(\cosupp) \le |V(\tilde\cosupp)|-t + s - 1
= |V(\tilde\cosupp)| - J(\tilde\cosupp) - t + s + 1
\le |V(\tilde\cosupp)| - J(\tilde\cosupp),
\]
which is what we wanted to show.
\end{proof}

For the next lemma, let us define the degree $\delta_\cosupp(v)$ of a vertex $v\in V(\cosupp)$:
\[
\delta_\cosupp(v) := | \{ e \in \cosupp : v \in e \} |
\]
where $v\in e$ signifies that $v$ is a vertex of the edge $e$.
That is, $\delta_\cosupp(v)$ is the number of edges in $\cosupp$ that start/end at $v$.

\begin{lemma}
\label{thm:cornerGain}
For a non-empty $\cosupp \subset E$,
\[
4|V(\cosupp)| \ge \sum_{v\in V(\cosupp)} \delta_\cosupp(v) + 4
\]
holds.
\end{lemma}

\begin{proof}
On a 2D regular grid, $\delta_\cosupp(v) \le 4$. Therefore, we have
\[
4|V(\cosupp)| \ge \sum_{v\in V(\cosupp)} \delta_\cosupp(v).
\]
Since the equality above would hold if and only if $\delta_\cosupp(v) = 4$ for all $v\in V(\cosupp)$,
the claim of the lemma can be proved by showing that there are at least two vertices $v$ with
$\delta_\cosupp(v) \le 2$.
For this, we consider two `extreme corner points' of $\cosupp$.
Let $v_{\mathrm{NW}}$ be the north-west corner point of $\cosupp$ in the sense that
1) there is no vertex $v\in V(\cosupp)$ that is above it, and 2) there is no
vertex $v \in V(\cosupp)$ that is left of $v_{\mathrm{NW}}$ \emph{and} on the same
level (height).
Let $v_{\mathrm{SE}}$ be the south-east corner point of $\cosupp$ defined similarly.
By definition, $\delta_\cosupp(v_{\mathrm{NW}}) \le 2$ and $\delta_\cosupp(v_{\mathrm{SE}}) \le 2$,
and $v_{\mathrm{NW}}$ and $v_{\mathrm{SE}}$ are distinct vertices if $\cosupp \not= \emptyset$.
\end{proof}

\begin{proof}[Proof of Proposition \ref{thm:Uniqueness2DTV}]
We will first prove the upper bound. 
Clearly,
\[
\sum_{v\in V(\cosupp)} \delta_\cosupp(v) = 2 |\cosupp|.
\]
By Lemma~\ref{thm:cornerGain}, we also have
\[
4|V(\cosupp)| \ge \sum_{v\in V(\cosupp)} \delta_\cosupp(v) + 4
\]
Hence, we have
\begin{equation}
\label{eq:vertexEdgeRelation}
|V(\cosupp)| \ge \frac{|\cosupp|}{2} + 1
\end{equation}
By Lemma~\ref{thm:oneComponent},
the value of $\dimmax_{\Oper_{\mathrm{DIF}}}(\cosparsity)$ given by Eq.~\eqref{eq:dimmaxGeneralGraph} is
attained for $\cosupp$ with $J(\cosupp) = 1$ and $|\cosupp| = \cosparsity$. 
Combining this with \eqref{eq:dimmaxGeneralGraph} and \eqref{eq:vertexEdgeRelation}
we get
\[
\dimmax_{\Oper_{\mathrm{DIF}}}(\cosparsity) \le |V|-(|V(\cosupp)|-1) \leq |V|-|\cosupp|/2 = \sdim - \frac{\cosparsity}{2}.
\]
The proof of the lower bound is given in Lemma~\ref{thm:lowerbound2dTV}.
\end{proof}

Before moving on to Lemma~\ref{thm:lowerbound2dTV}, we give a brief motivation for it.
Our goal is to obtain not just a lower bound on $\dimmax_{\Oper_{\mathrm{DIF}}}$ but a lower
bound that is close to optimal.
By Lemma~\ref{thm:oneComponent}, $\dimmax_{\Oper_{\mathrm{DIF}}}(\cosparsity)$ is achieved
for connected $\cosupp$, so we will consider such $\cosupp$'s only ($J(\cosupp) = 1$).
With $J(\cosupp) = 1$, the formula~\eqref{eq:dimmaxGeneralGraph} tells us to look for
the cases when $|V(\cosupp)|$ is minimal in order to compute $\dimmax_{\Oper_{\mathrm{DIF}}}(\cosparsity)$.

What is the shape of the collection of edges $\cosupp$ yielding the minimum ? Recalling Euler's formula for graphs on plane:
\begin{equation}
|V(\cosupp)| - |\cosupp| + |F(\cosupp)| = 2,
\end{equation}
where $F(\cosupp)$ is the faces of $\cosupp$ which includes the `unbounded one', we see that we are seeking $\cosupp$ such that $|F(\cosupp)|$ is maximal, i.e.,  there is maximum number of faces.
By intuition, we conjecture that this happens when $\cosupp$ consists of all the
edges in an almost square, by which we mean $V(\cosupp)$ is an $r\times r$ or
$r\times (r+1)$ rectangular grid or the inbetweens (e.g., an $r\times r$ grid of pixels to which $1 \leq j \leq r$ pixels have been added on one side).
These considerations lead to the following:

\begin{lemma}
\label{thm:lowerbound2dTV}
\[
\dimmax_{\Oper_{\mathrm {DIF}}}(\cosparsity) \ge \sdim - \frac{\cosparsity}{2} - \sqrt{\frac{\cosparsity}{2}} - 1
\]
for $\cosparsity \ge 5$.
\end{lemma}

\begin{proof}
For $r\ge2$, we consider a subgraph corresponding to an $r\times r$ square (solid lines)
and consider graphs obtained by adding additional edges in the fashion depicted in 
Figure~\ref{fig:2dTVsquare}.
\begin{figure}[htbp]
\begin{center}
\includegraphics[width=0.3\textwidth]{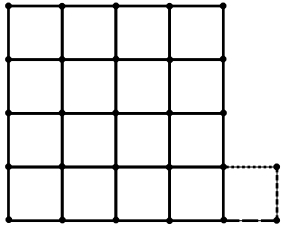}
\caption{Add dashed edges (from longer to shorter dashed) to $r\times r$ square subgraph (solid lines).}
\label{fig:2dTVsquare}
\end{center}
\end{figure}

We find that for the square $\cosupp$, $|\cosupp| = 2(r^2-r)$ and $|V(\cosupp)| = r^2$,
for the graph $\cosupp$ with one additional edge, $|\cosupp| = 2(r^2-r) + 1$ and 
$|V(\cosupp)| = r^2 + 1$,
for the graph $\cosupp$ with two additional edges, $|\cosupp| = 2(r^2-r) + 2$ and
$|V(\cosupp)| = r^2 + 2$,
and for the graph $\cosupp$ with three additional edges, $|\cosupp| = 2(r^2-r) + 3$ and
$|V(\cosupp)| = r^2 + 2$.
In fact, we observe that two edges can be added while adding one additional vertex
until $\cosupp$ corresponds to $r\times(r+1)$ rectangle. 
Summarizing all these, a graph $\cosupp$ that is constructed as above, is contained
$r\times(r+1)$ rectangle (included), and contains $r\times r$ square; satisfies either
$|\cosupp| = 2(r^2 - r) + 2j$ or $|\cosupp| = 2(r^2 - r) + 2j+1$, and
$|V(\cosupp)| = r^2 + j + 1$, for $j=1,\ldots,r-1$---this holds for $j=r$ as well.
(Here, the case $|\cosupp| = 2(r^2-r) + 1$ is not stated.)
By a similar observation, we observe that a graph $\cosupp$ that is constructed similarly
as above, is contained in $(r+1)\times(r+1)$ square (included), and contains $r\times(r+1)$
square; satisfies either
$|\cosupp| = 2r^2-1 + 2j$ or $|\cosupp| = 2r^2 - 1 + 2j+1$, and
$|V(\cosupp)| = r^2 + r + j + 1$, for $j=1,\ldots, r$---this holds for $j=r+1$ as well.

The above observation leads to the following inequalities---which we conjecture to be in fact
equalities:
\begin{align*}
\dimmax_{\Oper_{\mathrm {DIF}}}(2(r^2 - r) + 2j) &\ge \sdim - (r^2 + j),  \quad j=1,\ldots,r, \\
\dimmax_{\Oper_{\mathrm {DIF}}}(2(r^2 - r) + 2j+1) &\ge \sdim - (r^2 + j),  \quad j=1,\ldots,r, \\
\dimmax_{\Oper_{\mathrm {DIF}}}(2r^2-1 + 2j) &\ge \sdim - (r^2 + r + j), \quad j=1,\ldots,r+1, \\
\dimmax_{\Oper_{\mathrm {DIF}}}(2r^2-1 + 2j + 1) &\ge \sdim - (r^2 + r + j), \quad j=1,\ldots,r+1.
\end{align*}
We will now express these in a simpler form in terms of $|\cosupp| = \cosparsity$.
In the first case, letting $\cosparsity = 2(r^2 - r) + 2j$, we have
\[
\sdim - (r^2 + j) = \sdim - \frac{\cosparsity}{2} - r.
\]
Since 
\[
2(r^2-2r+1) \le 2(r^2 -r + 1) \le \cosparsity \le 2r^2,
\]
we have $r-1 \le \sqrt{\frac{\cosparsity}{2}} \le r$.
Hence, we can write $\dimmax_{\Oper_{\mathrm {DIF}}}(\cosparsity) \ge \sdim - \frac{\cosparsity}{2} - \sqrt{\frac{\cosparsity}{2}} - 1$.
The other three cases can be treated similarly and we obtain
\begin{align*}
\dimmax_{\Oper_{\mathrm {DIF}}}(\cosparsity) &\ge \sdim - \frac{\cosparsity}{2} - \sqrt{\frac{\cosparsity}{2}}, \\
\dimmax_{\Oper_{\mathrm {DIF}}}(\cosparsity) &\ge \sdim - \frac{\cosparsity}{2} - \sqrt{\frac{\cosparsity}{2}} - \frac{1}{2}, \\
\dimmax_{\Oper_{\mathrm {DIF}}}(\cosparsity) &\ge \sdim - \frac{\cosparsity}{2} - \sqrt{\frac{\cosparsity}{2}}
\end{align*}
Therefore, for all $\cosparsity\ge 5$, we have 
$\dimmax_{\Oper_{\mathrm {DIF}}}(\cosparsity) \ge \sdim - \frac{\cosparsity}{2} - \sqrt{\frac{\cosparsity}{2}} - 1$.
\end{proof}

\section{Discussion on the analysis exact recovery condition}
\label{sec:discussionERC}

We observe that the analysis ERC condition \eqref{eq:AL1ERC} is not sharp in general, especially for
the redundant $\Oper$.
In the case of GAP, tracing the arguments of Lemma~\ref{thm:GAPCoefRel} and Theorem~\ref{thm:GAPrecovery}, we conclude that in order for \eqref{eq:AL1ERC} to be sharp, there must exist a cosparse
signal $\x_0$ such that $\Oper_{\cosupp^c}\hat\x_0$ matches the exact sign pattern of the row
of $(\OMeas \Oper_\cosupp^T)^\dagger \OMeas \Oper_{\cosupp^c}^T$ with the largest $\ell_1$-norm
\emph{and} is of constant magnitude in absolute value. We remind 
that $\hat\x_0$ is the initial estimate that appears in the algorithm.
Since the collection of $\Oper_{\cosupp^c}\hat\x_0$ may not span the whole $\RR^{\cosupp^c}$, especially when $\Oper$ is over-complete,
it is unreasonable to expect the existence of such an $\x_0$.
Similarly, in the case of analysis $\ell_1$, we know that \eqref{eq:AL1ERC} is obtained 
from \eqref{eq:L1NSC} in a crude way without taking into account the sign patterns
of $\Oper_{\cosupp^c}\x_\cosupp$, which is not sharp in general for redundant $\Oper$.

\paragraph{Average case performance guarantees?}
Can we think of a way to obtain a more realistic success guarantee? 
We have a partial answer for this question in the sense that we can derive a condition---which is
not a guarantee---that reflects empirical results more faithfully.
The idea is, instead of obtaining an upper bound of the left hand side of \eqref{eq:L1NSC}
by disregarding (or considering the worst case of) sign patterns, to model the effects
of the sign patterns by estimating the size of the left hand side in terms of the maximum 
$\ell_2$-norm of the rows of $\left(\OMeas \Oper_\cosupp^T\right)^{\dagger} \OMeas \Oper_{\cosupp^c}^T$ (up to some constants). 
Though further investigation is desirable, we have empirically observed that the condition 
derived in this way reflected better the success rates of GAP and $\ell_1$-minimization.

\paragraph{Desirable properties for $\Oper$ and $\Meas$}
At this point, one may ask a practical question: what are desirable properties of $\Oper$
and $\Meas$ that would help the performance of GAP or $\ell_1$-minimization?
Can we gain some insights from our theoretical result?
For this, we look for scenarios where the entries of 
$\mathbf{R}_0 := \Oper_\cosupp \OMeas^T \left(\OMeas \Oper_\cosupp^T \Oper_\cosupp \OMeas^T\right)^{-1} \OMeas \Oper_{\cosupp^c}^T$ are small (hence, 
it is likely that condition~\eqref{eq:AL1ERC} is satisfied).
We start with the inner expression $\left(\OMeas \Oper_\cosupp^T \Oper_\cosupp \OMeas^T\right)^{-1}$.
The larger the minimum singular value of $\OMeas \Oper_\cosupp^T \Oper_\cosupp \OMeas^T$,
the smaller the entries of $\mathbf{R}_0$. 
First, assuming that the rows of $\Oper$ are normalized, we note that the minimum singular value
is larger when the size $\cosupp$ is larger. 
Second, the closer the minimum singular value is to the maximum one (this is in some sense
an RIP-like condition for $\Oper$), the larger it is.
These two observations tell us that $\Oper$ should have high linear dependencies (to allow large
cosupport $\cosupp$) and the rows of $\Oper$ should be close to uniformly distributed on
$\Sphe^{\sdim-1}$.

Suppose that $\Oper$ has the properties described above. Then, $\mathbf{R}_0$ is well approximated
by $\mathbf{R}_1 := \gamma \Oper_\cosupp \OMeas^T \OMeas \Oper_{\cosupp^c}^T$ for some $\gamma>0$.
Therefore, we ask when the entries of $\Oper_\cosupp \OMeas^T \OMeas \Oper_{\cosupp^c}^T$ are
small. 
Each entry of $\Oper_\cosupp \OMeas^T \OMeas \Oper_{\cosupp^c}^T$ can be guaranteed to be small
if a) $\OMeas$ satisfies an RIP condition for the space spanned by two rows of $\Oper$
and the rows of $\Oper$ are incoherent. 
In summary, it is desirable that:

\begin{itemize}
\item The rows of $\Oper$ are close to uniformly distributed in $\Sphe^{\sdim-1}$.

\item $\Oper$ is highly redundant and have highly linearly dependent structure.

\item $\Meas$ is `independent' from $\Oper$. This has to do with the RIP-like properties.

\item The rows of $\Oper$ are incoherent.

\item The cosparsity $\cosparsity$ are large.
\end{itemize}

\begin{remark}
The 2D finite difference operator $\Oper_{\mathrm{DIF}}$ may be considered incoherent
even though the coherence is relatively large ($1/4$). This is because the majority of pairs
of rows of $\Oper_{\mathrm{DIF}}$ are in fact uncorrelated.
\end{remark}

\paragraph{Heuristic comparison of success guarantees for analysis-$\ell^{1}$ and GAP}
We point out that one can obtain from \eqref{eq:CoefRel} a condition for the GAP that is similar to \eqref{eq:L1NSC}.
For this, we observe from \eqref{eq:CoefRel} that
\begin{align*}
\norm{\infty}{\Oper_\cosupp\hat\x_0} 
&= 
\norm{\infty}{(\OMeas\Oper_{\cosupp}^T)^{\dagger} \OMeas \Oper_{\cosupp^c}^T\Oper_{\cosupp^c}\hat\x_0 }
=
\norm{\infty}{
[\Oper_{\cosupp^c} \OMeas^{T}(\Oper_{\cosupp} \OMeas^T)^{\dagger}]^{T}\Oper_{\cosupp^c}\hat\x_0}.
\end{align*}
Since $\norm{\infty}{\Oper_\cosupp\hat\x_0} < \norm{\infty}{\Oper_{\cosupp^c}\hat\x_0}$ is
the necessary and sufficient condition for the (one-step) success of the GAP, we can derive
a necessary and sufficient condition:
\[
\norm{\infty}{[\Oper_{\cosupp^c} \OMeas^T (\Oper_{\cosupp}\OMeas^T)^{\dagger}]^{T}\Oper_{\cosupp^c}\hat\x_{0}} 
< \norm{\infty}{\Oper_{\cosupp^c}\hat\x_{0}}
\]
where $\x_0$ is varied over all signals with cosupport $\cosupp$ and $\hat\x_0$ is the signal resulting from the first step of GAP .
The above condition can be rewritten in a form similar to \eqref{eq:L1NSC}:
\begin{equation}
\label{eq:GAPNSC}
\sup_{\x_{0}}\norm{\infty}{[\Oper_{\cosupp^c} \OMeas^T (\Oper_{\cosupp}\OMeas^T)^{\dagger}]^{T} (\sign(\Oper_{\cosupp^c}\hat\x_{0})\odot \v)}  < 1
\end{equation}
where $\v := |\Oper_{\cosupp^c}\hat\x_{0}|/\norm{\infty}{\Oper_{\cosupp^c}\hat\x_{0}}$, i.e.,
$\v$ is obtained from $\Oper_{\cosupp^c}\hat\x_{0}$ by taking element-wise absolute values
and normalizing it to a unit $\ell_\infty$-norm,
and $\odot$ denotes the element-wise multiplication of vectors.
Condition~\eqref{eq:GAPNSC} and \eqref{eq:L1NSC} are in a similar form, but there are two differences between the two: First, for \eqref{eq:GAPNSC}, the signal $\hat\x_{0}$ that apears is not
in general a vector with cosupport $\cosupp$. It is rather a signal that arises from an approximation. Second, there is a `weight' vector $\v$ in \eqref{eq:GAPNSC}. One can heuristically deduce that
such a $\v$ favours condition \eqref{eq:GAPNSC} to hold true since the size of most entries of $\v$
likely be smaller than $1$.
Beside these differences, one should keep in mind that condition~\eqref{eq:GAPNSC} is only for one step.

\bibliographystyle{plain}
\bibliography{analysisVSsynthesis}
\end{document}